\documentclass[11 pt]{article}

\usepackage{amssymb} %use \usepackage[psamsfonts]{amssymb} if your fonts are from Y&Y/Blue Sky Research
\usepackage{amsmath} %use \usepackage{amsmath,cmmib57} if your fonts are from Y&Y/Blue Sky Research
\usepackage{latexsym}
\usepackage{graphicx}
\usepackage{amsthm,thmtools}
\usepackage{enumerate}

\graphicspath{ {images/} }
\usepackage{mathtools}

\providecommand{\keywords}[1]{\textbf{\textit{Keywords:}} #1}
     
\providecommand{\msc}[1]{{\textit{2000 Mathematics Subject Classification:}} #1}
\usepackage{tikz}
\usetikzlibrary{calc} 
\usepackage{pgf,tikz}
\usetikzlibrary{arrows}
\usepackage{amscd}
\usepackage{relsize}
   
\usepackage{tikz-cd}

\definecolor{NoteColor}{rgb}{1,0,0}



\usepackage[hidelinks]{hyperref}
\usepackage{caption, subcaption}
\usepackage[pdf]{pstricks}
\usepackage{makeidx}

\usepackage[affil-it]{authblk}

 \usepackage[capitalize]{cleveref}
\newtheorem{theorem}{Theorem}[section]
\newtheorem{proposition}[theorem]{Proposition}
\newtheorem{lemma}[theorem]{Lemma}

\newtheorem{corollary}[theorem]{Corollary}

\theoremstyle{definition}
\newtheorem{definition}[theorem]{Definition}

\newcommand{\R}{\mathbb R}
\let\reals\R
\newcommand{\naturals}{\mathbb N}

\renewcommand{\v}{\mathbf{v}}
\newcommand{\len}{\mathrm{length}}
\newcommand{\ie}{i.e.\ }

\theoremstyle{remark}
\newtheorem{remark}[theorem]{\rm\bf Remark}

\newtheorem{example}[theorem]{\rm\bf Example}

\begin{document}

\title{On families of Finsler metrics
% with applications to Funk and Hilbert geometries and to Teichmüller spaces
}
\author{\.{I}smail Sa\u{g}lam, Ken'ichi Ohshika  and Athanase Papadopoulos}

\date{}	
	\maketitle

	\begin{abstract}
		In this paper, we answer some natural questions on symmetrisation and more general combinations of Finsler metrics, with a view towards applications to Funk and Hilbert geometries and to metrics on  Teichmüller spaces. For a general non-symmetric Finsler metric on a smooth manifold, we introduce two different families of metrics, containing as special cases the arithmetic and the max symmetrisations respectively of the distance functions associated with these Finsler metrics. We are interested in various natural questions concerning metrics in such a family, regarding its  geodesics, its completeness, conditions under which such a metric is Finsler, the shape of its unit ball in the case where it is Finsler, etc. We address such questions in particular  in the setting of Funk and Hilbert geometries, and in that of the Teichmüller spaces of several kinds of surfaces, equipped with  Thurston-like asymmetric metrics.

			\end{abstract}
\noindent 	\keywords{} Symmetrisation of a Finsler metric, Hilbert metric, weighted Funk metric, families of Finsler metrics, complete non-symmetric metric, Teichm\"uller space,  Teichm\"uller spaces of Euclidean surfaces, Thurston metric.
	 	
\noindent 	\msc{} primary 53C70; secondary 51K05; 51K10;  53B40; 53C60

	\tableofcontents

\section{Introduction}

In this paper, we study some natural questions concerning symmetrisations and other combinations of metrics on a given space, in particular of Finsler metrics. Finsler geometry is considered here from a synthetic point of view, with a minimum of differentiability properties and without use of tensor calculus. The motivation for writing this paper comes from the activity  taking place since a few years in geometry and  topology, in particular in the study of the Funk and Hilbert metrics (see the papers 	
\cite{AFV, PTF, PY2, RV}), of Thurston's asymmetric Finsler metric on Teichm\"uller space together with its adaptations to various settings \cite{Thurston, HOP,  MOP, SP1, SP2, SOP},  and of the recent work on the Finsler geometry of the earthquake metric defined on the same space \cite{HOPP}.  The family of Randers metrics connecting the Teichm\"uller metric on the Teichm\"uller space of the torus to the Thurston-like asymmetric metric of that space  studied in \cite{MOP-2022} and its generalisation to higher genera in \cite{MOP-2024} also acted as a motivation for studying families of general Finsler metrics. In writing this paper, we felt the need to clear up in a general setting some results on the metric geometry of Finsler manifolds.  We adopt a point of view of global Finsler geometry, following in the footsteps of Herbert Busemann, who wrote in
 \cite{Busemann1950-1}: ``The term `Finsler space' evokes in most mathematicians the picture
of an impenetrable forest whose entire vegetation consists of
tensors. The purpose of the present lecture is to show that the association
of tensors (or differential forms) with Finsler spaces is due to an
historical accident, and that, at least at the present time, the fruitful
and relevant problems lie in a different direction.'' 
%This is also the point of view adopted in the papers \cite{PT22}  and \cite{PT2}. 

Thus, in the following, a Finsler manifold $M$ is defined to be a differentiable manifold equipped with a weak norm on each tangent space such that these norms vary continuously with respect to the basepoints in $M$.
Here and in the following,  the epithet ``weak" means that the structure under consideration does not necessarily satisfy the symmetry axiom. It will be understood that the metrics and norms we consider are weak, without always stating this explicitly.
 In a Finsler manifold, the length of a piecewise $C^1$ path is defined as the integral of the norms of the tangent vectors along this path. With this, a Finsler manifold is equipped with a natural length structure and a metric in which the distance between two points is equal to the infimum of the lengths of all  piecewise $C^1$ paths joining them.
   In this setting,  we shall address natural questions concerning the symmetrisation and more general combinations of Finsler metrics. We obtain results on the geodesics of these metrics, their Finsler structures, and the associated families of unit balls and Minkowski functionals.  
 
The outline of the rest of this paper is as follows.

In \cref{s:Examples}, we give four examples which prepare the reader for the kind of questions we address in the rest of the paper. The examples are concerned with the comparison between metrics obtained from two given Finsler metrics by taking the maximum of the distance functions, or the max of the norms, or the sum of the distance functions, or the sum of the norms. We give examples and counter-examples which will show the kind of answers we expect. This will prepare for the general theorems which we obtain later in the paper.

 In \cref{s:Weak}, we recall the notions of weak metric, weak length space, geodesic space, non-symmetric metric satisfying an axiom of Busemann formulated in his book \cite{Busemann-GG},  and other metric notions, adapted to our non-symmetric setting. 
  
In \cref{s:Weak-Finsler},  
we collect several results on Finsler structures, some of which exist in the literature but under stronger regularity conditions.
We prove that the metrics induced by our Finsler structures satisfy Busemann's axioms introduced in \cref{s:Weak-Finsler}.
We discuss notions like Cauchy sequence, metric completion and uniform convergence  for our setting of asymmetric metric spaces and we give an Arezel\`{a}--Ascoli-type theorem as well as a Hopf--Rinow theorem adapted to this setting, together with applications to the study of two natural one-dimensional families of Finsler metrics associated with a given Finsler metric $F$, namely,

\[
F_t^a(x,v)=(1-t)F(x,v)+tF(x,-v),
\] and
\[
F_t^m(x,v) = \max\{((1-t) F(x,v),t F(x,-v)\}, both
\]
 defined for $t\in [0,1]$

In \cref{s:Families}, 
given a Finsler structure $F$ on a manifold $M$ with its induced distance function $d(F)$, we give conditions under which an element of the families of  distance functions 
\[
d(F)_t^a(x,y)=(1-t) d(F)(x,y)+t d(F)(y,x)
\]
and
\[
d(F)_t^m(x,y)=\max\{(1-t)d(F)(x,y), td(F)(y,x)\}
\] 
 $(t\in [0,1])$, is Finsler, and in the case it is, a formula for the Lagrangian of this Finsler structure in terms of that of $F$.  
  Furthermore, for any two Finsler structures, we  give a necessary and sufficient condition for a weighted Finsler structure obtained from them by the formulae we gave in \cref{s:Weak-Finsler} to induce the same distance function as the weighted distance function with the same weight.
   The questions raised by the examples given in  \cref{s:Examples} are answered in a general setting. Some of the results are a broad generalisation of results that are known for Hilbert and Funk metrics.

In \cref{s:Weighted-Funk}, we introduce two families of metrics, the arithmetic and max weighted Funk metrics. The Euclidean straight line segments are geodesics for the arithmetic weighted Funk metrics. Thus, these metrics satisfy the conditions of Hilbert's Fourth Problem. 

 \cref{s:Teich} is a summary of results and questions related to Finsler metrics on Teichm\"uller spaces.

\section[Four examples]{Four examples: Maximum of norms versus maximum of distances and sum of norms versus sum of distances}\label{s:Examples}
In this section, we give four illuminating examples which prepare for the kind of questions we address in the rest of this paper.
All manifolds are assumed to be differentiable. The metrics  involved in these examples are symmetric, unlike many of the metrics that we consider in the rest of this paper.

To simplify notation, in the following, given a Finsler structure on a manifold, we shall usually denote by  $\Vert\cdot \Vert$  the family of associated norms on the tangent bundle and also the individual norms on each tangent space.

Given two Finsler metrics $d_1$ and $d_2$ on a manifold $M$ induced by two families of norms $\Vert\cdot \Vert_1$ and  $\Vert \cdot \Vert_2$ on the tangent bundle of $M$, we consider two ways of defining new metrics by taking maxima: 
one way is to take the norm $\max\{\Vert\cdot\Vert_1, \Vert\cdot \Vert_2\}$ at each point, and then take the Finsler metric and the distance function defined by this new norm, and the other way is to take the maximum $\max\{d_1, d_2\}$ of the distance functions. 
The reader will notice that it is not clear a priori whether  the latter metric is induced by a norm, that is, whether  the manifold equipped with this distance function is Finsler. 

We start with an example which shows that the two metrics obtained by taking the max of norms and distances respectively do not coincide in general, even if the original metrics are Riemannian  (\cref{ex:1} below).
In the second example (\cref{ex:2}), we show, by starting with the same original metrics, that  in this example, the metrics  obtained by taking respectively the sum of the of norms and distances do not coincide.
 In the next example (\cref{ex:3}), the two metrics, obtained by taking the max of the distance functions and of the norms, coincide.
Later in the paper, we shall study the general case, and in particular we shall address the question of whether the metric $\max\{d_1, d_2\}$ is Finsler.
  \cref{ex:4} is concerned with the comparison between the metrics induced by the sum of the norms  $\lvert\lvert \ \rvert\rvert_1+ \lvert\lvert \ \rvert\rvert_2$ and the one defined by the sum of the distance functions $d_1+d_2$.  In the example we give, the two metrics coincide.

\begin{example}\label{ex:1}[Unequal metrics obtained as max of distances and max of norms.] Let $U =\{(x,y)\in \reals^2\mid y>0\}$ be the upper half-plane.
We consider the following two norms defined on $U$: the Euclidean norm $\displaystyle\Vert \cdot \Vert_e=\sqrt{dx^2+dy^2}$ and the hyperbolic norm $\displaystyle\Vert\cdot \Vert_h=\frac{\sqrt{dx^2+dy^2}}{y}$.
We denote the corresponding Riemannian metrics on $U$ by $d_e$ and $d_h$ respectively.

Let $\Vert \cdot \Vert_m$ be the norm defined on each tangent plane of $U$ by $\Vert \v \Vert_m=\max\{\Vert \v \Vert_e, \Vert \v \Vert_h\}$.
From the definition, it follows that
 for any tangent vector $\v$ at $(x,y)$ with $y \geq 1$, we have $\Vert\v \Vert_m=\Vert \v \Vert_e$ whereas for $y <1$, we have $\Vert \v \Vert_m=\Vert \v \Vert_h$.
Let $d_m$ be the Finsler metric on $U$ induced by the norm $\Vert \cdot \Vert_m$.

Consider two points of the form $(0, y_1), (0,y_2) \in U$ with $y_1 < 1 < y_2$.
Then, we have $d_e((0,y_1), (0,y_2))=y_2-y_1$ and $d_h((0,y_1), (0,y_2))=\log(y_2/y_1)$.
We shall prove that $d_m((0,y_1), (0,y_2))=(y_2-1)-\log(y_1)$. We can easily find some specific pair of points $y_1$ and $y_2$ so that the values taken by the metrics $d_m$ and $\max\{d_e, d_h\}$ on this pair are different.

The inequality $d_m((0,y_1), (0,y_2)) \leq (y_2-1)-\log(y_1)$ is immediate since the length of the vertical segment joining the two points with respect to $\Vert \cdot \Vert_m$ is equal to the expression given in the right hand side.

\sloppy
Let us prove the reverse inequality.
Let $\alpha \colon [a,b] \to U$ be an arc connecting $(0,y_1)$ to $(0,y_2)$, and let us denote its coordinates by $\alpha(t)=(x(t),y(t))$.
Then there exists  $t_0 \in [a,b]$ such that $y(t_0)=1$, and
 $\len_{d_m}(\alpha[a, t_0]) \geq \len_{d_h}(\alpha[a,t_0])\geq - \log(y_1)$.
On the other hand, $\len_{d_m}(\alpha[t_0, b]) \geq \len_{d_e}(\alpha[t_0,b]) \geq y_2-1$.
Thus we have $\len_{d_m}(\alpha) \geq (y_2-1)-\log y_1$. Taking the infimum over all arcs joining $(0,y_1)$ and $(0,y_2)$, we have the required inequality, $d_m((0,y_1), (0,y_2)) \geq (y_2-1)-\log y_1$. Thus, $d_m((0,y_1), (0,y_2))=(y_2-1)-\log(y_1)$. 
\end{example}

\begin{example}\label{ex:3}[Equal metrics obtained by taking the max of the distances and of the norms.]  
Consider the following two Riemannian metrics on $\R^2$: 
$$ds^2=dx^2+dy^2 \ \text{and}\ ds'^2=a dx^2+b dy^2,$$

\noindent where $a$ and $b$ are positive real numbers. These metrics are induced by the following norms respectively: $\lvert\lvert \ \lvert\lvert_e=\sqrt{dx^2+dy^2} $
and $\lvert\lvert \ \lvert\lvert_{e'}=\sqrt{adx^2+b dy^2}$. 

It is easy to see that $(\R^2, ds)$ and $(\R^2,ds')$ are isometric as Riemannian manifolds. Indeed, an isometry is given by $$(x,y)\mapsto (\frac{1}{\sqrt{a}}x,\frac{1} {\sqrt{b}}y).$$
Since this isometry  is linear, the geodesics of the metric 
$(\R^2, ds')$ are, like those of  $(\R^2, ds)$, straight lines.

Let $d_e$ and $d_{e'}$ be the length metrics associated with $\lvert \lvert \cdot \rvert \rvert_e$ and $\lvert\lvert \cdot \rvert\rvert_{e'}$ respectively, that is, the metrics where the distance between two points with respect to $d_e$ (respectively $d_{e'}$) is the infimum of the lengths of the piecewise $C^1$-paths joining these two points with respect to the norm $\Vert \cdot \Vert_e$ (respectively $\Vert \cdot \Vert_{e'}$). 
Then we have $$d_e((x,y),(x',y'))=\sqrt{(x'-x)^2+(y'-y)^2}$$ and 
$$d_{e'}((x,y),(x',y'))=\sqrt{a(x'-x)^2+b(y'-y)^2}.$$

Consider the norm $\max\{\Vert\cdot \Vert_e,\Vert\cdot \Vert_{e'}\}$, and let $\mu$
be the induced Finsler metric on $\R^2$. 
Consider also the metric $m=\max\{d_e,d_{e'}\}$. 
We shall prove that $\mu=m$. 

It is easy to see, as follows, that $\mu \geq m$.
For any two points $A$ and $B$ in $\R^2$ and for any $\epsilon>0$, there exists a piecewise $C^1$-path $\alpha:[0,1]\to \R^2$, with $\alpha(0)=A$ and $\alpha(1)=B$, such that  
$$\mu(A,B)\geq \int_0^1\max\{\lvert\lvert\dot{\alpha}(t)\rvert\rvert_e,\lvert\lvert\dot{\alpha}(t)\rvert\rvert_{e'}\}\ dt-\epsilon \geq \int_0^1\lvert\lvert\dot{\alpha}(t)\rvert\rvert_e\ dt -\epsilon\geq d_e(A,B)-\epsilon.$$
It follows that $\mu(A,B)\geq d_e(A,B)$. By the same reasoning, $\mu(A,B)\geq d_{e'}(A,B)$. 
Thus we have the inequality $\mu\geq m.$

Now we prove the reverse inequality. Let $\alpha: [0,1]\to \R^{2}$ be the unique path such that $\dot{\alpha}$ is a constant vector, $\alpha(0)=A$ and $\alpha(1)=B$. 
Since $\dot{\alpha}$ is a constant vector, one of the inequalities $\Vert \dot{\alpha} \Vert_e \geq \Vert \dot{\alpha}\Vert_{e'}$ or $\Vert\dot{\alpha}\Vert_{e} < \Vert \dot{\alpha}\Vert_{e'}$ holds all over $[0,1]$.
Assume that the first alternative holds, \ie $\max\{\lvert\lvert \dot{\alpha}\rvert\rvert_{e} , \lvert\lvert \dot{\alpha}\rvert\rvert_{e'}\}=\lvert\lvert \dot{\alpha}\rvert\rvert_{e} $. Then,
 
$$d_e(A,B)=\int_0^1\lvert\lvert\dot{\alpha}(t)\rvert\rvert_e \ dt\geq \int_0^1\lvert\lvert\dot{\alpha}(t)\rvert\rvert_{e'}\ dt=d_{e'}(A,B).$$
 Thus $m(A,B)=d_e(A,B)$. Therefore we have
\begin{equation*}
\begin{split}m(A,B)&=d_e(A,B)=\int_0^1\lvert\lvert\dot{\alpha}(t)\rvert\rvert_e \ dt\\&=\int_0^1\max\{\lvert\lvert \dot{\alpha}(t)\rvert\rvert_e,\lvert\lvert \dot{\alpha}(t)\rvert\rvert_{e'} \} \ dt \geq \mu(A,B).
\end{split}
\end{equation*}

The same kind of argument works in   the case when the second alternative holds.
We have then  $\max\{\lvert\lvert \dot{\alpha}\rvert\rvert_{e} , \lvert\lvert \dot{\alpha}\rvert\rvert_{e'}\}=\lvert\lvert \dot{\alpha}\rvert\rvert_{e'}$, which implies $d_{e'}(A,B) \geq d_e(A,B)$ and hence $d_m(A,B)=d_{e'}(A,B)$.
Now by replacing $d_e$ and $\Vert\cdot\Vert_e$ in the above by $d_{e'}$ and $\Vert\cdot\Vert_{e'}$,   we  have $m \geq \mu$.
Thus we conclude that $m=\mu$

Note that the above argument also shows that straight lines are geodesics for the Finsler metric $m$.
\end{example}

 \cref{ex:3} is a special case of a more general result which we shall prove later, \cref{th:max}.

\begin{example}\label{ex:2}[Unequal metrics obtained as sums of distance functions and of norms.]
 Now we turn to the metric obtained by taking the sum of two norms, $\Vert \cdot \Vert_s:=\Vert\cdot \Vert_1+\Vert\cdot \Vert_2$. 
We start with the two metrics $d_e$ and $d_h$ defined in  \cref{ex:1} on the upper half-plane $U$.
Let $d_s$ denote the Finsler metric associated with the norm $\Vert \cdot \Vert_s$, and $d_\sigma$ the distance function defined as the sum $d_e+d_h$.
Take two points $a_1=(x_1, y_1)$ and $a_2=(x_2, y_2)$ in $U$ with $x_1 \neq x_2$.

The geodesic connecting $a_1$ to $a_2$ with respect to $d_e$ is a straight line segment, which we denote by $\gamma_1$.
Let $\gamma_2$ denote the geodesic connecting the same pair of points with respect to $d_h$.
Note that $\gamma_1$ and $\gamma_2$ meet only at their endpoints.
Now, by definition, $d_\sigma(a_1, a_2)=\len_{d_e}(\gamma_1)+\len_{d_h}(\gamma_2)$.
On the other hand, for any rectifiable arc $\alpha$ connecting $a_1$ to $a_2$, we have $\len_{d_s}(\alpha)=\len_{d_e}(\alpha)+\len_{d_h}(\alpha)$.

We have $\len_{d_e}(\alpha)\geq \len_{d_e}(\gamma_1)$ with equality holding only when $\alpha=\gamma_1$, and we have $\len_{d_h}(\alpha)\geq \len_{d_h}(\gamma_2)$ with equality holding only when $\alpha=\gamma_2$.
Since $\gamma_1 \neq \gamma_2$ in our setting,  we see that $\len_{d_s}(\alpha)>\len_{d_e}(\gamma_1)+\len_{d_h}(\gamma_2)=d_\sigma(a_1, a_2)$.
By the compactness of the space of rectifiable paths connecting $a_1$ and $a_2$ with lengths bounded from above (use the Arzel\`{a}--Ascoli theorem), this implies that $d_s(a_1, a_2) > d_\sigma(a_1, a_2)$.
\end{example}

 \begin{example} \label{ex:4}[Equal metrics obtained by taking the sum of the distances or of the norms.]
 In this example, we start with two Finsler (in fact, Riemannian) metrics, and show that in this particular case the metric defined by taking the sum of the two metrics coincides with the one obtained by taking the sum of the two norms.
 
We start with the same Riemannian metrics on $\R^2$ as in \cref{ex:3}.
 These metrics are induced by the norms $\lvert\lvert \cdot \lvert\lvert_e=\sqrt{dx^2+dy^2}$
and $\Vert \cdot \Vert_{e'}=\sqrt{a dx^2+b dy^2}$ respectively. 
As before, we denote the associated length metrics by $d_e$ and $d_{e'}$. Consider the norm $\Vert \cdot \Vert_e+ \Vert \cdot \Vert_{e'}$ on $\R^2$ and let $\eta$ be the induced Finsler metric.
We define a metric $\delta$ by  $\delta=d_{e}+d_{e'}$. We shall prove that $\eta=\delta$.

First we observe that $\eta\geq \delta$. Indeed, since $\eta$ is Finsler, for any two points  $A$ and $B$  in $\R^2$ and for all $\epsilon>0$, there exists a  piecewise $C^1$-curve $\alpha: [0,1]\to \R^2$ from $A$ to $B$ such that
\begin{align*}
\eta(A,B)&\geq \int_0^1\{\lvert\lvert\dot{\alpha}(t)\rvert\rvert_e+\lvert\lvert\dot{\alpha}(t)\rvert\rvert_{e'}\}\ dt-\epsilon
\\ &= \int_0^1\lvert\lvert\dot{\alpha}(t)\rvert\rvert_e\ dt+\int_0^1\lvert\lvert\dot{\alpha}(t)\rvert\rvert_{e'}\ dt -\epsilon\geq d_e(A,B)+d_{e'}(A,B)-\epsilon\\
&=\delta(A,B)-\epsilon.
\end{align*}
 It follows that $\eta(A,B)\geq \delta(A,B)$.  
 
 Now we prove the reverse inequality. Let $\alpha:[0,1]\to \R^2$ be the unique path such that $\dot{\alpha}$ is a constant vector, $\alpha(0)=A$ and $\alpha(1)=B$. Then we have
\begin{align*}
\delta(A,B)&=d_{e}(A,B)+d_{e'}(A,B)=\int_0^1\lvert\lvert\dot{\alpha}(t)\rvert\rvert_e\ dt+\int_0^1\lvert\lvert\dot{\alpha}(t)\rvert\rvert_{e'}\ dt \\
&=\int_0^1\{\lvert\lvert\dot{\alpha}(t)\rvert\rvert_e+\lvert\lvert\dot{\alpha}(t)\rvert\rvert_{e'}\}\ dt\geq \eta(A,B).
\end{align*}

Thus we conclude that $\eta=\delta$. 
The argument above also implies that straight lines are geodesics for the metric $\delta$.

\end{example}

 \cref{ex:4} illustrates the meaning of \cref{th:sum} below. 

\section{Weak metrics and weak length spaces}\label{s:Weak}
We adopt the point of view on metric spaces developed in \cite{PT22}, and we first recall the following notion.

\begin{definition}[Weak metric]
A weak metric on a set $X$ is a function $\eta: X \times X\to [0,\infty]$ such that 

\begin{enumerate}
\item
$\eta(x,x)=0$ for all $x \in X$.
\item
$\eta(x,z)\leq \eta(x,y)+\eta(y,z)$ for all $x,y,z \in X$.
\end{enumerate}
\end{definition}

Thus, a weak metric does not necessarily satisfy the symmetry axiom.  It is interesting to note that Felix Hausdorff, one of the main founders of the theory of metric spaces, already introduced a metric that does not satisfy the symmetry axiom in his famous book \emph{Mengenlehre} (Set Theory) \cite{Hausdorff2}, first published in 1927, which is a classical treatise on topology and metric spaces. Indeed,  Hausdorff defined there  the distance  between bounded subsets of a metric space that now bears his name, the \emph{Hausdorff distance}, but before introducing this metric, he considered a non-symmetric version of it (see p. 167 of the English translation of \cite{Hausdorff2}). We also note that this non-symmetric Hausdorff distance  was used by Thurston in his paper \cite{Thurston}, and it was studied in an essential way in the paper \cite{OP1}.
The Funk metric, introduced by Funk in \cite{Funk} and studied by Busemann in his book \cite{Busemann-GG},   whose theory was later  developed in several directions including generalisations to convex sets in non-Euclidean spaces of constant curvature and in a Lorentzian setting  (see e.g. \cite{PY, PY2, PY3}), is another classical example of weak (non-symmetric) metric.
The study of non-symmetric metrics has recently been the object of growing interest for geometers, in particular since the appearance of Thurston's metric on Teichm\"{u}ller space, see \cite{Thurston} and the recent paper \cite{HOP}, and more recently, of  the earthquake metric, also introduced by Thurston in \cite{Thurston} and developed in \cite{HOPP}.  Some works by Busemann on completion and other properties of non-symmetric metric spaces have recently been  made more precise  and widely extended for their use in specific cases of non-symmetric metrics: see for example the  paper \cite{HOPP} and especially the appendix there, and the paper \cite{SOP}.

In order to spare words, and if the context is clear, we shall sometimes write \emph{metric} instead of \emph{weak metric}.

Since our metrics are not necessarily symmetric, there are a priori two possibilities of defining the convergence of a sequence $(x_n)$ to a point $x$ in the space:   $\eta(x_n,x)\to 0$ or  $\eta(x,x_n)\to 0$  as $n\to \infty$. Busemann, in \cite{Busemann-Synthetic}, introduced the following axiom which guarantees that these two possibilities are equivalent.

\begin{definition}[Busemann's axiom]
\label{Busemann}
	Let $(X,\eta)$ be a weak metric space.
	 We say that $(X,\eta)$ satisfies Busemann's axiom if the following holds: For any sequence $(x_n)$ and for any $x$ in X, we have $\eta(x_n,x)\to 0$ as $n\to \infty$
	if and only if $\eta(x,x_n)\to 0$ as $n\to \infty$. 
\end{definition}

In a space satisfying Busemann's axiom, the topologies obtained by taking as a sub-basis for open sets the left or the right open balls or their union are the same. These topologies coincide with the one generated by the \lq\lq arithmetically symmetrised'' weak metric $\eta_{arith}(x,y)=\frac{1}{2}[\eta(x,y)+\eta(y,x)]$.

In what follows, we shall assume that the spaces we consider satisfy Busemann's axiom. In this case, they are equipped with a natural  topology.

Given a weak metric $\eta$ on a metric space $X$, we are interested in the following families of metrics, each of them parametrised by $t\in [0,1]$:
\begin{equation}
\eta_t^a(x,y)=(1-t)\eta(x,y)+t\eta(y,x)
\end{equation}
\begin{equation}
\eta_t^m(x,y)=\max\{(1-t)\eta(x,y), t\eta(y,x)\}.
\end{equation}  
In these formulae and others which will come later, the superscripts $a$ and $m$ stand for ``arithmetic" and ``max".  

We recall a notion from \cite{PT2}.
Let $X$ be a topological space. A collection $\Gamma$ of continuous paths $\gamma: [a,b]\to X$, where $[a,b]$ is a compact interval of $\R$, is called a {\it semigroupoid
of paths} on $X$ if the following two properties are satisfied:
\begin{enumerate}
\item
If $\gamma_1:[a,b]\to X$ and $\gamma_2:[c,d]\to X$ are two paths in $\Gamma$ with $\gamma_1(b)=\gamma_2(c)$, then the concatenation $\gamma_1*\gamma_2$ also lies 
in $\Gamma$.
\item
Any constant path belongs to $\Gamma$. 
\end{enumerate}
A semigroupoid of paths satisfies the usual axioms of a semigroupoid in the algebraic (or categorical) sense.

\begin{definition}[Weak length structure]
Let $X$ be a topological space and $\Gamma$ a semigroupoid of paths on $X$. A weak length structure on $(X,\Gamma)$ is a function 
$l:\Gamma \to [0,\infty]$ such that the following conditions are satisfied.
\begin{enumerate}
\item
If $\gamma_1$ and $\gamma_2$ are in $\Gamma$ and $\gamma_1*\gamma_2$ is defined and is in $\Gamma$, then $l(\gamma_1*\gamma_2)=l(\gamma_1)+l(\gamma_2)$.
\item
For any constant path $c$, we have $l(c)=0$.
\item
If $f: [c,d]\to [a,b]$ is a continuous surjective non-decreasing function and $\gamma: [a,b]\to X$ is in $\gamma$, then $l(\gamma\circ f)=l(\gamma)$ whenever $\gamma\circ f\in \Gamma$.
\end{enumerate}
\end{definition}
%\begin{definition}
%	Let $(X,\eta)$ be a weak metric space. We say that $(X,\eta)$ satisfies Busemann's axiom if the following holds: $\eta(x_n,x)\to 0$ as $n\to \infty$
%	if and only if $\eta(x,x_n)\to 0$ as $n\to \infty$ for all sequences in $(x_n)$ and for all $x$ in X. For such a space, the topology on $X$ is the topology generated by symmetric weak metric $\eta_{arith}(x,y)=\frac{1}{2}[\eta(x,y)+\eta(y,x)]$.
%\end{definition}

\begin{definition}[Weak length structure for a weak metric space]
\label{rectifiable}
	Let $(X,\eta)$ be a weak metric space. (Recall that all our spaces satisfy Busemann's axiom so that they are equipped with the topology defined after \cref{Busemann}.) Let $c:[a,b]\to X$ be a continuous path in $X$. We define $l(\eta)(c)$ as the supremum of all sum of the form 
	
	$$\sum_{i=0}^n d(f(t_i),f(t_{i+1})),$$
	\noindent where $a=t_0\leq t_1\leq t_{n+1}=b$ is a finite partition of $[a,b]$. Then $l(\eta)$ is a weak length structure on the semigroupoid of continuous paths on $X$.
\end{definition}

\begin{definition}[Weak length space]
A weak length space  is a triple $(X,\Gamma,l)$, where $X$ is a topological space, $\Gamma$ a semigroupoid of continuous paths on $X$, and $l$ a weak length structure on $(X,\Gamma)$.
\end{definition}

There is a natural weak metric on a weak length space:
\begin{definition}[Weak metric associated with a weak length space]
If $(X,\Gamma,l)$ is a weak length space, then the weak metric induced by $l$  on $X$ is the function on $X\times X$ defined, for each $x,y\in X$, by
$$\delta(l)(x,y)=\inf\{l(\gamma)\mid \gamma \in\Gamma_{x,y}\}$$
 where $\Gamma_{x,y}$ is the subset of $\Gamma$ consisting of paths joining $x$ to $y$.
 If $\Gamma_{x,y}$ is empty, we define $\delta(l)(x,y)$ to be $\infty$.
\end{definition}
Let $\Gamma$ be a semigroupoid of paths on a weak length space $(X,\Gamma, l)$ such that for each $\gamma \in \Gamma$, the inverse path $\gamma^{-1}$ of $\gamma$ is also in $\Gamma$. We define a family of length structures on $X$, parametrised by $t\in [0,1]$, by the formula
\begin{equation}\label{eq:family}
l_t(\gamma)=(1-t)l(\gamma)+tl(\gamma^{-1}).
\end{equation}

\begin{lemma}
\label{first-inequality}
Let $(X,\Gamma,l)$ be a weak length space, and $\delta(l)$ the induced weak metric on $X$. Then  for all $x,y \in X$ and for all $t\in [0,1]$, the following inequality holds:
$$\delta(l_t)(x,y)\geq \delta(l)_t^a(x,y):=(\delta(l))_t^a(x,y).$$

\begin{proof}
We can assume that $\delta(l_t(x,y)) \neq \infty$ since the inequality trivially holds in the case when $\delta(l_t(x,y))=\infty$.
For any $\epsilon >0$, we can find an element $\gamma\in \Gamma_{x,y}$ satisfying
\begin{align*}
\delta(l_t)(x,y)\geq l_t(\gamma)-\epsilon &= (1-t)l(\gamma)+tl(\gamma^{-1})-\epsilon\\
&\geq (1-t)\delta(l)(x,y)+t\delta(l)(y,x)-\epsilon\\
&=\delta(l)_t^a(x,y)-\epsilon.
\end{align*}
Since this inequality is valid for any $\epsilon >0$, the result follows.
\end{proof}
\end{lemma}

\begin{definition}[Minimal and biminimal paths, geodesic space]
\label{def:geodesic}
Let $(X,\Gamma,l)$ be a weak length structure, and $\delta(l)$ the induced weak metric. A continuous path $\gamma \in \Gamma_{x,y}$ is said to be minimal (or a geodesic)
if $l(\gamma)=\delta(l)(x,y)$ and $l(\gamma) \neq \infty$. The path $\gamma$ is called bi-minimal (or a bi-geodesic) between $x$ and $y$ if $\gamma$ is minimal, $\gamma^{-1}\in \Gamma_{y,x}$ and $\gamma^{-1}$ is also  minimal, that is,
$l(\gamma^{-1})=\delta(l)(y,x)$.  We say that  $X$ is a geodesic space if there is a minimal path between any two points of $X$.
\end{definition}

\begin{proposition}
\label{first-equality}
Let $(X,\Gamma,l)$ be a weak length structure, and $\delta(l)$ the induced weak metric.
Assume that  $\Gamma$ is a semigroupoid of continuous paths on $X$ such that for each $\alpha \in \Gamma$, its inverse $\alpha^{-1}$ is in $\Gamma$.
Consider the family of length structures on $X$ parametrised by $t\in [0,1]$ defined by \cref{eq:family},
and let $x$ and $y$
be two points in $X$ such that there exists a bi-minimal path $\gamma \in \Gamma$ joining $x$ to $y$.  Then
$$\delta(l_t)(x,y)= \delta(l)_t^a(x,y).$$
\end{proposition}
\begin{proof}
By \cref{first-inequality}, we only need to prove the inequality 
$$
\delta(l_t)(x,y)\leq\delta(l)_t^a(x,y).$$

We can see this as follows:
\begin{align*}
\delta(l)_t^a(x,y)&=(1-t)\delta(l)(x,y)+t\delta(l)(y,x)\\
&=(1-t)l(\gamma)+tl(\gamma^{-1})\\
&\geq \inf\{ (1-t) l(\alpha)+tl(\alpha^{-1}): \alpha \in \Gamma_{x,y}\}\\
&=\delta(l_t)(x,y).
\end{align*}
\end{proof}

Now let $l$ and $l'$ be two length structures on $(X,\Gamma)$, and $\gamma$ a path in $\Gamma$. For each $t\in [0,1]$, we define the following length structure:
\begin{equation}
[(1-t)l+tl'](\gamma)=(1-t)l(\gamma)+tl'(\gamma).
\end{equation}

If $\eta$ and $\eta'$ are two weak metrics on $X$, then clearly $(1-t)\eta+t\eta'$ and $\max\{(1-t)\eta, t\eta'\}$ are weak metrics on $X$ as well.

\begin{lemma}
\label{average}
Let $l$ and $l'$ be two weak length structures on $(X,\Gamma)$. Then the following inequality holds:
$$\delta((1-t)l+tl')(x,y)\geq[(1-t)\delta(l)+t\delta(l')](x,y).$$
\end{lemma}
\begin{proof}
We can assume that $\delta(1-t)l+tl')(x,y) \neq \infty$, for otherwise the inequality trivially holds.
For every $\epsilon >0$, we can find an element $\gamma \in \Gamma_{x,y}$ satisfying
\begin{align*}
\delta((1-t)l+tl')(x,y)&\geq (1-t)l(\gamma)+tl'(\gamma)-\epsilon\\\
&\geq (1-t)\delta(l)(x,y)+t\delta(l')(x,y)-\epsilon,
\end{align*}
 from which the result follows.
\end{proof}

\begin{proposition}
\label{average-maximum}
Let $l$ and $l'$ be two weak length structures on $(X,\Gamma)$. Assume that $x,y \in X$ are two points such that there exists $\gamma\in \Gamma_{x,y}$ 
which is minimal for both $l$ and $l'$. Then, for each $t\in [0,1]$, we have
\[\delta((1-t)l+tl')(x,y)=[(1-t)\delta(l)+t\delta(l')](x,y).\]
\end{proposition}
\begin{proof}
By \cref{average}, we only need to prove the inequality
$$\delta((1-t)l+tl')(x,y)\leq[(1-t)\delta(l)+t\delta(l')](x,y).$$

This can be seen as follows:
\begin{align*}
[(1-t)\delta(l)+t\delta(l')](x,y)&=(1-t)l(\gamma)+tl'(\gamma)\\
 &\geq\inf\{(1-t)l(\alpha)+tl'(\alpha):\alpha \in \Gamma_{x,y}\}\\
 &= \delta((1-t)l+tl')(x,y).
 \end{align*}
\end{proof}

\section{Finsler Structures}\label{s:Weak-Finsler}

We start with the general notion of Finsler structure.
Given a smooth manifold $M$, we denote the tangent space at $x \in M$ by  $T_xM$, and the  tangent bundle by $TM$.

\begin{definition}[Finsler Structure]\label{def:weak-Finsler}
A Finsler structure on a smooth manifold $M$  is a continuous function $F: TM\to [0,\infty)$ such that for every point $x$ in $M$, the restriction of $F$ to 
$T_xM$, $F(x,\cdot):=F\lvert_{T_xM}$, is a weak norm, that is, a function which satisfies the following properties:
\begin{enumerate}
\item
$F(x,v_1+v_2)\leq F(x,v_1)+F(x,v_2)$ for all $v_1,v_2 \in T_xM$.
\item
$F(x,\lambda v)=\lambda F(x,v)$ for all $\lambda \geq 0$ and for all $v \in T_xM$.
\item
$F(x,v)=0$ if and only if $v=0$.
\end{enumerate}
\end{definition}
In other words, a Finsler structure is a continuous family of weak  norms defined on the tangent spaces of $M$.

 \begin{definition}[The weak length structure associated with a Finsler structure]
Let $F$ be a Finsler structure on a smooth manifold $M$. Then $F$ defines a weak length structure $l_F$ whose associated semi-groupoid of paths is  the set of piecewise $C^1$-paths on $M$ and where the length of a piecewise $C^1$-path $\beta: [a,b]\to M$ is given by the quantity
\begin{equation}
\label{eq:int}
l_F(\beta)=\int_{a}^{b}F(\beta(s),\dot{\beta}(s)) \ ds.
\end{equation}
\end{definition}

The function $F$ in the above definition is called the \emph{Lagrangian} of the Finsler structure $F$. 
 The metric $d(F)(x,y)=\delta(l_F)(x,y))$ associated with the induced weak length structure $l_F$ is called the \emph{metric induced by the Finsler structure $F$}.  A metric on $M$ is called \emph{Finsler} (or \emph{Finslerian}) 
if it is induced by some Finsler structure.

%The length can be defined for continuous paths in  general asymmetric metric spaces as follows, in the same way as in \cref{rectifiable}.
%\begin{definition}
%Let $(X,d)$ be an asymmetric metric space.
%Let $\alpha \colon [a,b] \to X$ be a continuous path.
%The length of $\alpha$ is defined to be
%\begin{equation*}
%l_m(\alpha):=\sup_{a=t_0< t_1 \dots < t_n=b} \sum_{k=0}^{n-1} d(\alpha(t_k, t_{k+1}),
%\end{equation*}
%where $\sup$ is taken over all subdivisions of $[a,b]$.
%A path $\alpha$ is said to be rectifiable when $l_m(\alpha)$ is finite.
%\end{definition}
%
%A  Finsler structure is said to be {\em geodesic} if it is a geodesic space  in the sense of Definition \ref{def:geodesic} as a weak length structure.

 \begin{definition}[$C^{\infty}$-Finsler structure]
	A Finsler structure on a smooth manifold $M$ is called a {\it $C^{\infty}$-Finsler structure} if $F$ is a $C^{\infty}$ function on $TM-\{0\}$, where $TM-\{0\}$ is the complement of the zero section in the tangent bundle $TM$ of $M$.

 \end{definition}
 
We now prove that the metrics induced by Finsler structures satisfy the Busemann  axiom  introduced in \cref{Busemann}.
We note that this property, under a stronger assumption on the regularity of $F$, is mentioned in \S6 of \cite{BCS}.

\begin{proposition}
\label{Finsler Busemann}
Let $F$ be a Finsler structure on a smooth manifold $M$, and let $d(F)$ be the associated metric on $M$. Let $(x_i)$ be a sequence in $M$.
Then the following three conditions are equivalent:
\begin{enumerate}[(i)]
\item $d(F)(x_i,x)\to0$, 
\item$d(F)(x,x_i) \to 0$, and 
\item $(x_i)$ converges to $x$ with respect to the topology of $M$.
\end{enumerate}
In particular, Busemann's axiom holds for $d(F)$.
\end{proposition}

To prove this proposition, we first show the following elementary lemma.
\begin{lemma}
\label{comparison}
Let $V$ be a finite-dimensional vector space over $\reals$, and let $\Vert \cdot \Vert_1$ and $\Vert \cdot \Vert_2$ be two weak norms on $V$.
Then there is a positive constant $c$ such that for any $v$ in $V$, we have $\Vert v \Vert_2 \leq c \Vert v \Vert_1$.
\end{lemma}
\begin{proof}
Fixing a basis of $V$, we can define a ``standard" Euclidean norm on $V$, which we denote by $\Vert \cdot \Vert$.
We have only to prove the inequalities in both directions between $\Vert \cdot \Vert$ and positive scalar multiples of $\Vert \cdot \Vert_2$.
Let $S$ be the unit sphere of $V$ with respect to $\Vert \cdot \Vert$.
Since $S$ is compact and $\Vert \cdot \Vert_2$ is continuous, $\sup_{s \in S} \Vert s \Vert_2$ is finite and $\inf_{s\in S}\Vert s \Vert_2$ is positive, which we set to be $c_1$ and $c_2$ respectively.
Then for any $v \in V$, we have $\displaystyle\Vert v\Vert_2=\Vert v\Vert \Vert \frac{v}{\Vert v\Vert}\Vert_2$ and $\displaystyle c_2 \Vert v\Vert \leq \Vert v\Vert\Vert \frac{v}{\Vert v\Vert}\Vert_2\leq c_1 \Vert v\Vert$, and we are done.
\end{proof}

\begin{proof}[Proof of \cref{Finsler Busemann}]
It is sufficient to prove the equivalence between (i) and (iii), for this also implies the equivalence between (ii) and (iii).

We fix some Riemannian metric $g$ on $M$.
We denote by $d$ the distance function induced from $g$ and by $\Vert \cdot \Vert$ the  norm associated with $g$.
We shall prove that $d(F)(x_i, x) \to 0$ if and only if $d(x_i,x) \to 0$.
By \cref{comparison}, there exists a constant $c$ such that $\displaystyle c^{-1} \leq \frac{\Vert v \Vert }{F(x,v)} \leq c$ for all $v \in T_xM \setminus \{\mathbf 0\}$.
Moreover, since  $F(x, v)$ is continuous with respect to $x \in M$, then for any $x \in M$, there is neighbourhood $U$ of $x$ and a constant $K \geq 1$ such that 
\begin{equation}K^{-1} \leq \frac{\Vert v\Vert}{F(y,v)} \leq K \label{comparison with d}\end{equation}
 for every $y \in U$ and every $v \in T_y M$.

Now, let $(x_i)$ be a sequence in $M$ with $d(F)(x_i, x) \to 0$.
For any small $\epsilon >0$, the $\epsilon$-neighbourhood $U_d(\epsilon)$ of $x$ with respect to $d$ is contained in $U$.
For all $y \in M \setminus U_d(\epsilon)$, any arc connecting $x$ to $y$ has a part contained in $U$ whose length with respect to $d$ is at least $\epsilon$.
By \cref{comparison with d}, for any $y \in M \setminus U_d(\epsilon)$, any piecewise $C^1$-path connecting $x$ to $y$ has length with respect to $F$ at least $K^{-1} \epsilon$.
Since $d(F)(x_i, x) \to 0$, there is $i_0 \in \naturals$ such that if $i \geq i_0$, then $x_i$ is contained in $U_d(\epsilon)$.
Since $\epsilon$ is an arbitrarily small positive number, this shows that $(x_i)$ converges to $x$ with respect to $d$.

To show the reverse implication, let $(x_i)$ be a sequence in $M$ with $d(x_i, x)\to 0$.
For any $\epsilon >0$, there is $i_0$ such that if $i \geq i_0$, then $x_i$ is contained in $U$ and can be joined to $x$ by a piecewise $C^1$-arc $\alpha$ with $d$-length less than $\epsilon$.
Then $\displaystyle d(F)(x_i, x) \leq \int_\alpha F(\alpha(t), \dot{\alpha}(t))dt \leq \int_\alpha K\Vert \dot{\alpha}(t)\Vert dt< K\epsilon$.
Thus we see that $d(F)(x_i, x) \to 0$.
\end{proof}
%\begin{remark}
%	Let $F$ be a Finsler structure on a manifold $M$. Assume that the induced weak metric $d(F)$ satisfies Busemann's axiom. Then we can consider the weak length structure $l(d(F))$. It is not difficult to see for a piecewise 
%	$C^1$ path $c:[a,b]\to M$ the following equality holds:
%	
%	$$l(d(F))(c)=l_F(c).$$
% \end{remark}

We now start to study geodesics and completeness in our Finsler spaces.

Since $(M, d(F))$ is a weak metric space satisfying Busemann's axiom, we can define the length for any continuous path as in \cref{rectifiable}.
We denote the length of a path $\alpha$ in this sense by $l({d(F)})(\alpha)$.
We say that a continuous path is {\em rectifiable} if its length is finite.
A result of Busemann--Mayer shows that the two notions of length are equivalent for piecewise $C^1$-paths \cite{BM}:
\begin{lemma}
\label{B-M}
For any piecewise $C^1$-path $\alpha \colon [a,b] \to M$, we have $l({d(F)})(\alpha)=l_F(\alpha)$.
\end{lemma}
This has been generalised to absolutely continuous paths, see \cite{Burt,ZZ}.
We recall that an arc $\gamma \colon [a,b] \to M$ is said to be absolutely continuous if for any $\epsilon >0$, there exists $\delta>0$ such that for any disjoint intervals $(a_1, b_1), \dots , (a_n,b_n)$ with total length less than $\delta$, we have $\sum_{i=1}^n d(\gamma(a_i), \gamma(b_i)) < \epsilon$.
It is well known that an absolutely continuous path $\gamma(t)$ has derivative at almost every $t$, and hence that its length $l_F(\gamma)$ is well defined.
The same argument as in the proof of \cref{B-M} works for absolutely continuous paths, and we have the following (see \cite{ZZ}).
\begin{lemma}
\label{two lengths}
For any absolutely continuous path  $\gamma \colon [a,b] \to M$, we have $l({d(F)})(\gamma)=l_F(\gamma)$.
\end{lemma}

\begin{definition}[Geodesics and bi-geodesics]
	Let $F$ be a Finsler structure on a manifold $M$. 
	We say that a path $\gamma: [a,b]\to M$ is a geodesic (or minimal) with respect to  $F$ if it is continuous, rectifiable and $l(d(F))(\gamma)=d(F)(c(a),c(b))$.
	  We say that a path $\gamma: [a,b]\to M$ is a bi-geodesic (or bi-minimal) if it and its reverse are geodesics.
\end{definition}

As a corollary to \cref{two lengths}, we have the following.
\begin{corollary}
\label{longer}
For any absolutely continuous path $\gamma \colon [a,b] \to M$ we have $l_F(\gamma) \geq d(F)(\gamma(a), \gamma(b))$, and the equality holds if and only if $\gamma$ is a geodesic.
\end{corollary}

In the following section, we shall need to study the special case when any pair of points can be joined by an absolutely continuous geodesic.
%Since any absolutely path $\gamma$ is a uniform limit of $C^1$-path $\gamm_i$ in such a way that $l(d(F))(\gamma_i) \to l(d(F))(\gamma)$, (see \cite{Burt} and \cite{ZZ}), we see that there is 

\begin{definition}
We say that a Finsler structure on $M$ is {\em geodesic} if every two points on $M$ can be joined by an {\em absolutely continuous} geodesic.
\end{definition}

We next turn to considering the notion of completeness with respect to $d(F)$.
For that, we first need to define Cauchy sequences for weak metrics.

\begin{definition}[Cauchy sequences and completeness]
A sequence $(x_i)$ in a weak metric space $(X,\eta)$ is said to be a {\em Cauchy sequence} with respect $\eta$ if for any $\epsilon >0$, there exists $N$ such that for any $n\geq m \geq N$, we have $\eta(x_m, x_n) < \epsilon$.
The weak metric $\eta$ is said to be {\em complete} if every Cauchy sequence $(x_i)$ has a limit in $X$, \ie if $\eta(x_i, x) \to 0$ for some $x \in X$.
A Finsler structure $F$ on $M$ is said to be complete when the corresponding weak metric $d(F)$ is complete.
\end{definition}
These conditions are referred to as \emph{forward Cauchyness} and  \emph{forward completeness} respectively, whereas the condition obtained by replacing $n> m$ by $m>n$ is called \lq\lq backward Cauchyness''  and \lq\lq backward completeness'' in some literature (see \cite{CZ}).

\begin{lemma}
\label{hereditary}
Let $F_1$ and $F_2$ be two complete Finsler structures on $M$.
Then both $F_1+F_2$ and $\max\{F_1, F_2\}$ are also complete.
\end{lemma}
\begin{proof}
Let $(x_i)$ be a (forward) Cauchy sequence for $d(F_1+F_2)$.
Since $d(F_k)(x_m, x_n) \leq d(F_1+F_2)(x_m, x_n)$ for $k=1,2$, the sequence $(x_i)$ is a Cauchy sequence for both $d(F_1)$ and $d(F_2)$.
Since we assumed that both $F_1$ and $F_2$ are complete, $(x_i)$ converges  with respect to both $d(F_1)$ and $d(F_2)$.
By \cref{Finsler Busemann}, the limit for $d(F_1)$ and that for $d(F_2)$ are both the limit with respect to the topology of $M$.
Therefore they must coincide. We denote their limit by $y$.
The conditions  $d(F_1)(x_i, y) \to 0$ and $d(F_2)(x_i, y) \to 0$ imply that $d(F_1+F_2)(x_i, y) \to 0$, and we see that $(x_i)$ is convergent with respect to $d(F_1+F_2)$, which means that $F_1+F_2$ is complete.
The same argument works also for $\max\{F_1,F_2\}$.
\end{proof}
%
%We shall now prove that $y_1=y_2$.
%Since $\eta_1(y_1, y_2) \leq \eta_1(y_1, x_m) + 
%\end{proof}

By virtue of \cref{Finsler Busemann}, an analogue of  the Arzel\`{a}--Ascoli theorem (\cref{Arzela-Ascoli}) holds, as was shown in Theorem 5.12 in \cite{CZ}.
As a consequence, we have the following \cref{Hopf-Rinow}, which is an analogue of Hopf--Rinow's theorem in our setting.

\begin{proposition}
\label{Hopf-Rinow}
Suppose that $F$ is a complete Finsler structure on $M$.
Then for any $x, y \in M$, there is a geodesic with respect to $F$ connecting $x$ to $y$.
\end{proposition}

Before starting the proof, we need to introduce the notion of forward/ backward uniform convergence.

\begin{definition}[Uniform convergence]
Let $(c_i \colon [a,b] \to M)$ be a sequence of continuous paths.
We say that $(c_i)$ forward (resp. backward) uniformly converges to a continuous path $c\colon [a,b] \to M$ if  for any $\epsilon$, there is $i_0$ such that the inequality $d_F(c_i(t), c(t))<\epsilon$ (resp.\ $d_F(c(t), c_i(t))< \epsilon$) holds for every $i \geq i_0$ and $t \in [a,b]$.
\end{definition}

\cref{Finsler Busemann} implies the following.

\begin{lemma}
\label{both uniform}
A sequence $(c_i \colon [a,b]\to M)$ as above forward uniformly converges to a continuous path $c\colon [a,b]\to M$ if and only if it backward uniformly converges to $c$.
\end{lemma}
\begin{proof}
We shall prove that forward uniform convergence of $(c_i)$ to $c$ implies the backward uniform convergence.
The other way round can be proved in exactly the same way.
Since both $c_i$ and $c$ are continuous with respect to  the topology of $M$  (\cref{Finsler Busemann}), and $d(F)$ is continuous on $M$, then for any $i$, there is $t_i \in [a,b]$ such that $d(F)(c(t_i), c_i(t_i))$ attains the $\max_{t \in [a,b]}\{d(F)(c(t), c_i(t))\}$.
We only need to show that $d(F)(c(t_i), c_i(t_i)) \to 0$ as $i \to \infty$.
Since $(c_i)$ forward uniformly converges to $c$, we see that $d(F)(c_i(t_i), c(t_i)) \to 0$ as $i \to \infty$.
It suffices to prove that under this condition, there is always a subsequence such that $d(F)(c(t_i), c_i(t_i)) \to 0$ as $i \to \infty$.

Passing to a subsequence, we can assume that the sequence $(t_i)$ converges to $t_\infty$.
Then, by the continuity of $c$,  we have  $d(F)(c(t_i) c(t_\infty)) \to 0$, and we also have $d(F)(c_i(t_i), c(t_\infty))\to 0$ by our assumption of  forward uniform convergence and the triangle inequality.
Now applying \cref{Finsler Busemann}, we have $d(F)(c(t_\infty), c_i(t_i)) \to 0$, and by the triangle inequality, $d(F)(c(t_i), c_i(t_i)) \to 0$.
\end{proof}

We now state the Arezel\`{a}--Ascoli theorem in a form adapted to our purpose.
(The parametrisation is proportional to arc-length and the length bound imply the (forward) equicontinuity used in \cite{CZ}.)

\begin{proposition}
\label{Arzela-Ascoli}
Let $F$ be a complete Finsler structure on $M$.
Let $(\alpha_i \colon [a, b] \to M)$ be a sequence of piecewise $C^1$-paths parametrised proportionally to arc-length such that for any $m \leq n$  and for all $t \in [a,b]$, there exists a constant $K$ with $d(F)(\alpha_m(t), \alpha_n(t))\leq K$,  and  for all $i$, $l_F(\alpha_i) \leq L$.
Then, up to passing to a subsequence, $(\alpha_i)$ converges uniformly to a continuous arc.
\end{proposition}

\begin{proof}[Proof of \cref{Hopf-Rinow}]
By the definition of $d(F)$, there exists a sequence of piecewise $C^1$-curves $(\alpha_i\colon [a,b] \to M)$  connecting $x$ to $y$ such that $d(F)(x,y)=\lim_{i\to \infty}l_F(\alpha_i)$.
Since their lengths converge, there is a constant $L_0$ such that $l_F(\alpha_i) \leq L_0$ for all $i$.
We parametrise the $\alpha_i$ proportionally to arc length.
%Since $l_F(\alpha_i)$ converges, we can assume (by restricting the intervals of definition) that all of $\alpha_i$ are maps from a common interval $[0,\ell]$.
For $t \in [a,b]$ and $j \geq i$, we have $d(F)(\alpha_i(t), \alpha_j(t)) \leq d(F)(\alpha_i(t), y)+d(F)(y,x)+d(F)(x, \alpha_j(t)) \leq 2L_0+d(F)(y,x)$.
Therefore we can apply \cref{Arzela-Ascoli}. Passing to a subsequence,  $(\alpha_i)$ converges to a continuous arc $\alpha \colon [a,b] \to M$ uniformly.

What remains to show is that $\alpha$ is rectifiable and $l({d(F)})(\alpha) = d(F)(x,y)$.
The rectifiability and the inequality $l(d(F))(\alpha) \leq d(F)(x,y)$  follow from the lower semi-continuity of the length function in the case when $F$ is symmetric, and we can apply the same argument by virtue of \cref{both uniform}.
She other inequality $l({d(F)})(\alpha) \geq d(F)(x,y)$ is an immediate consequence of the definition of length. Thus, we have $l({d(F)})(\alpha)=d(F)(x,y)$, which means that $\alpha$ is a geodesic from $x$ to $y$.
\end{proof}

\begin{example}
A geodesic in a Finsler manifold is not necessarily piecewise $C^1$, as can be seen in the case when $M=\reals^2$ and $F(x,(a,b))=|a|+|b|$.	
\end{example}

Although \cref{Hopf-Rinow} is an analogue of Hopf--Rinow's theorem, we do not know whether geodesics we obtained by the proposition are differentiable or even absolutely continuous.
The followibg stronger result is known for $C^\infty$-Finsler structures. 

\begin{proposition}[Theorem 6.6.1 in \cite{BCS}]
\label{smooth Hopf-Rinow}
Suppose that $F$ is a  complete $C^\infty$-Finsler structure on $M$.
Then any two points on $M$ can be joined by a $C^\infty$-geodesic.
\end{proposition}

Next, we turn to considering weighted sums and weighted max of either two arbitrary Finsler structures or a Finsler structure and its reverse.
Let  $F_1$ and $F_2$ be two Finsler structures on a manifold $M$. We are interested in  the following two families of Finsler structures obtained from $F_1$ and $F_2$, defined for $t\in [0,1]$: 
\begin{equation}
\label{two Finsler}
((1-t)F_1+tF_2)(x,v)=(1-t)F_1(x,v)+t F_2(x,v),
\end{equation}
\begin{equation}
\label{two Finsler max}
(\max\{(1-t)F_1,tF_2\})(x,v)=\max\{(1-t)F_1(x,v),t F_2(x,v)\}.
\end{equation} 

Concerning the first family, we have the following result:

\begin{proposition}
\label{crucial}
Let $F_1$ and $F_2$ be two Finsler structures on $M$.
% Then $(1-t)l(F_1)+tl(F_2)$, as a length structure, is associated with the Finsler structure $(1-t)F_1+tF_2$. In other words, 
Then for each absolutely continuous path
$\gamma: [0,1]\to M$, the following equality holds:
$$[(1-t)l_{F_1}+tl_{F_2}](\gamma)=l_{((1-t)F_1+t F_2)}(\gamma).$$
\end{proposition}
\begin{proof}
We have

\begin{align*}
l_{((1-t)F_1+tF_2)}(\gamma)&=\int_0^1 [(1-t)F_1(\gamma(s),\dot{\gamma}(s))+tF_2(\gamma(s),\dot{\gamma}(s))] \ ds\\
&=(1-t)\int_0^1F_1(\gamma(s),\dot{\gamma}(s)) \ ds+t\int_0^1F_2(\gamma(s),\dot{\gamma}(s))  ds\\ &=[(1-t)l_{F_1}+tl_{F_2}](\gamma).
\end{align*}
\end{proof}
%Even for general rectifiable paths, one direction of inequality holds:
%\begin{proposition}
%\label{rectifiable bound}
%In the same setting as in \cref{crucial}, for any rectifiable path $c \colon [a,b] \to M$, we have $$[(1-t)l_{F_1}+tl_{F_2}](\gamma)\leq l_{((1-t)F_1+t F_2)}(\gamma).$$
%\end{proposition}
%

As a special case, for a given  Finsler structure, we shak
 consider  two associated families of Finsler structures obtained from it, which we call weighted Finsler structures.

With a Finsler structure $F$ on a manifold $M$, we associate the following two families of  Finsler structures, defined for $t\in [0,1]$:
\begin{equation}\label{eq:A}
F_t^a(x,v)=(1-t)F(x,v)+tF(x,-v),
\end{equation}
\begin{equation}\label{eq:m}
F_t^m(x,v) = \max\{((1-t) F(x,v),t F(x,-v)\}.
\end{equation}
 
 As before, the superscripts $a$ and $m$ stand for ``arithmetic" and ``maximum". 
 We observe that the fact that $F_t^a$ and $F_t^m$ are Finsler structures follows from our definitions (\cref{def:weak-Finsler}). 
These can be regarded as special cases of \cref{two Finsler,two Finsler max}, in the case when $F_1=F$ and $F_2(x,v)=F(x,-v)$.
The following is an easy consequence of \cref{crucial}.
\begin{corollary}
\label{first-length-structure}
Let $F$ be a Finsler structure on $M$. For $t\in[0,1]$, consider the length structure $(l_F)_t$  obtained from $F$ using \cref{eq:int} and \cref{eq:family}.  Then, for every piecewise smooth path $\gamma:[a,b]\to M$, the following equality holds.
$$(l_F)_t(\gamma)=l_{F_t^a}(\gamma).$$
\end{corollary}
%\begin{proof}
%We  have
%\begin{align*}
%	l_{F_t^a }(\gamma) &= \int_a^b F_t^a(\gamma(s),\dot{\gamma}(s))\ ds \\
%	&=  \int_a^b\bigg((1-t)F(\gamma(s),\dot{\gamma}(s))+t F(\gamma(s),-\dot{\gamma}(s)) \bigg)\ ds \\     
%	&= (1-t)\int_0^1 F(\gamma(s),\dot{\gamma}(s)) \ ds + t\int_a^b F(\gamma(s),-\dot{\gamma}(s))\ ds \\
%	&= (1-t)\int_a^b F(\gamma(s),\dot{\gamma}(s)) \ ds + t\int_a^b F(\gamma(1-s),\dot{\gamma}(1-s))\ ds \\
%	&= (1-t)l_F(\gamma) + tl_F(\gamma^{-1})   =  (l_{F})_t(\gamma).
%\end{align*}
%\end{proof}

\section{Families of weighted Finsler structures}\label{s:Families}
In this section, for a given  Finsler structure $F$ on a manifold $M$ with its induced distance function $d(F)$, we find conditions under which the modified distance functions $d(F)_t^a$ and $d(F)_t^m$, for $t\in [0,1]$  are also Finsler. 
Furthermore, in the case when these distance functions are Finsler, we describe their associated Lagrangians. 
Given two Finsler structures, we also provide a necessary and sufficient condition for a weighted Finsler structure obtained from them as in \cref{eq:A} or {eq:m} to induce the same distance function as the weighted distance function with the same weight. This  general result includes the examples presented in \cref{s:Examples} in a broad setting. 
This resuly also includes in a much broader setting a result obtained in \cite{PT2} for the Hilbert and Funk metrics, namely, the characterisation of the Minkowski functional of a Finsler structure obtained as the arithmetic symmetrisation of the one we started with.

%We call a Finsler structure {\em $C^1$-geodesic} if any two points can be joined by a piecewise $C^1$-geodesic.
%This is a condition stronger than requiring that any two points should be joined by a rectifiable geodesic.
%However, in many important examples such as Funk metrics or several metrics on Teichm\"{u}ller spaces, this stronger condition is satisfied.

\begin{theorem}\label{th:sum}
Let $F_1$ and $F_2$ be two geodesic Finsler structures on $M$, and $t$  a number in $(0,1)$. 
Then the Finsler structure $[(1-t)F_1+t F_2]$ is a geodesic Finsler structure inducing the same metric as $(1-t)d(F_1)+td(F_2)$ if and only if  for every $x,y \in M$, there is a absolutely continuous path $\gamma$ joining $x$ to $y$ which is a geodesic with respect to both  $F_1$ and $F_2$. 
Furthermore, in this case, any absolutely continuous path that is a geodesic for both $F_1$ and $F_2$ is also a geodesic for $[(1-t)F_1+t F_2]$.
\end{theorem}
\begin{proof}
Suppose that $\gamma$ is an absolutely continuous geodesic with respect to both $F_1$ and $F_2$.
Using \cref{two lengths,crucial}, we see that 
$$l_{((1-t)F_1+t F_2))}(\gamma)= (1-t)l_{F_1}(\gamma)+ tl_{F_2}(\gamma),$$
and since $\gamma$ is a geodesic for both $F_1$ and $F_2$, the latter term coincides with $(1-t)d(F_1)(x,y)+td(F_2)(x,y)$.

Now we prove that $\gamma$ is geodesic under $(1-t)F_1+t F_2$. 
Let $\beta$ be any arc connecting $x$ and $y$.
Then by \cref{two lengths,crucial} again,
\begin{align*}
l(d((1-t)F_1+ t F_2))(\alpha)&=l_{(1-t)F_1+tF_2}(\alpha)\geq (1-t)l_{F_1}(\alpha)+t l_{F_2}(\alpha)\\
& \geq (1-t)l_{F_1}(\gamma)+t l_{F_2}(\gamma)\\
&=l_{((1-t)F_1+ t F_2)}(\gamma)=l(d((1-t)F_1+tF_2))(\gamma).
 \end{align*}
Therefore $\gamma$ is an absolutely continuous path shorter than any piecewise $C^1$-path connecting $x$ and $y$, which means that it is a geodesic.
%This also shows the last statement of our theorem.

%
%First, we shall prove the sufficiency.
%Let $\gamma \colon [a,b] \to M$ be a rectifiable path which is a geodesic for both for $F_1$ and $F_2$.
%Since $\gamma$ is a geodesic for both $F_1$ and $F_2$, for $j=1,2$ and any partition $a=t_0 < t_1 < \dots < t_{n+1}=b$, we have $d_{F_1}(\gamma(a), \gamma(b))=\sum_{i=0}^n d(F_j)(\gamma(t_i, t_{i+1}))=d(F_j)(\gamma(a), \gamma(b))$.
%Then we have $\sum_{i=0}^n ((1-t)d(F_1)+td(F_2))(\gamma(t_i), \gamma(t_{i+1}))=(1-t)d(F_1)(\gamma(a), \gamma(b))+td(F_2)(\gamma(a), \gamma(b))$.
%Considering all partitions and taking the supremum, we see that $\gamma$ is a geodesic with respect to $(1-t)d(F_1)+td(F_2)$.
%

%Using \cref{rectifiable bound}, we see that 
%$$l_{((1-t)F_1+t F_2))}(\gamma) \geq (1-t)l_{F_1}(\gamma)+ tl_{F_2}(\gamma),$$
%and since $\gamma$ is a geodesic for both $F_1$ and $F_2$, this coincides with $(1-t)d(F_1)(x,y)+td(F_2)(x,y)$.
%Now we prove that $\gamma$ is a geodesic under $(1-t)F_1+t F_2$. 
%Let $\beta$ be any arc connecting $x$ and $y$.
%Then 
%\begin{align*}
%l_{((1-t)F_1+ t F_2)}(\alpha)&\geq (1-t)l_{F_1}(\alpha)+t l_{F_2}(\alpha)\\
%& \geq (1-t)l_{F_1}(\gamma)+t l_{F_2}(\gamma)\\
%&=l_{((1-t)F_1+ t F_1)}(\gamma).
% \end{align*}
%Therefore $\gamma$ is the shortest path connecting $x$ and $y$, and we are done.
%This also shows the last statement of our theorem.

Next we turn to the necessity.
Suppose that $(1-t)F_1+t F_2$ is a  geodesic Finsler structure inducing $(1-t)d(F_1)+ t d(F_2)$.
Suppose, seeking a contradiction, that there are $x, y \in M$ such that there is no common absolutely continuous path connecting $x$ and $y$ which is a geodesic for both $F_1$ and $F_2$.
Let $\beta$ be an absolutely continuous geodesic connecting $x$ and $y$ with respect to $(1-t)F_1+t F_2$.
Then, $\beta$ is not a geodesic for either $F_1$ or $F_2$. Therefore we have $l_{((1-t)F_1+t F_2)}(\beta)=(1-t)l_{F_1}(\beta)+t l_{F_2}(\beta) > (1-t)d(F_1)(x,y)+ t d(F_2)(x,y)$, contradicting our assumption.
\end{proof}

\begin{corollary}
	Let $\Omega$ be an open convex subset of $\R^n$. Assume that $F_1$ and $F_2$
	are two Finsler structures on $\Omega$ such that straight line segments are  geodesic for both $d(F_1)$ and $d(F_2)$. Let $t\in [0,1]$. Then the metric $(1-t)d(F_1)+td(F_2)$ is Finsler with Lagrangian $(1-t)F_1+tF_2$ and straight line segments are geodesics for this metric as well.
	\end{corollary}
	\begin{proof}
		This immediately follows from Theorem \ref{th:sum}.
	\end{proof}

In the special case when we consider a Finsler metric and its inverse metric, we have the following.

\begin{theorem}
\label{main-theorem}
Let $F$ be a  Finsler structure on $M$. Suppose that for every $x,y \in M$ there is an absolutely continuous bi-geodesic  joining $x$ to $y$. Then, for any $t\in [0,1]$, we have  
$$d(F)_t^a(x,y)= \delta(l_{(F_t^a)})(x,y).$$
 In particular $d(F)_t^a$ is a Finsler metric with Lagrangian $F_t^a$. Furthermore, every absolutely continuous bi-geodesic for $F$ is a bi-geodesic also for $F_t^a$.
 \end{theorem}
\begin{proof}
Let $\gamma$ be an absolutely continuous bi-geodesic joining $x$ to $y$. Using \cref{first-length-structure}, we see that 
\begin{equation*}
\delta(l_{(F_t^a)})(x,y)=\inf\{(l_F)_t^a(c)|c \in \Gamma_{x,y}\}=\inf\{(1-t)l_F(c)+tl_F(c^{-1})\mid c \in \Gamma_{x,y}\}.
\end{equation*}
The right hand side term is bounded from below as
\begin{equation}
\begin{split}
&\inf\{(1-t)l_F(c)+tl_F(c^{-1})\mid c \in \Gamma_{x,y}\}\\
&\geq (1-t)\inf\{l_F(c)\mid c\in \Gamma_{x,y}\}+t\inf\{l_F(c^{-1})\mid c \in \Gamma_{x,y}\}\\
&=(1-t)l_F(\gamma)+tl_F(\gamma^{-1})\\ 
&= (1-t)d(F)(x,y)+td(F)(y,x)\ (\text{by \cref{two lengths}}) \ =d(F)_t^a(x,y) .
\end{split}
\label{arithmetic t}
\end{equation}
On the other hand, again by \cref{two lengths}, we have
\begin{equation*}
\inf\{(1-t)l_F(c)+tl_F(c^{-1})\mid c \in \Gamma_{x,y}\}\leq (1-t)l_F(\gamma)+tl_F(\gamma^{-1})=d(F)_t^a(x,y).
\end{equation*}
Thus, $\delta(l(F_t^a)=d(F)_t^a$.
The last equality of \cref{arithmetic t} shows that both $\gamma$ and $\gamma^{-1}$ are geodesics for $F_t^a$.
%Now we prove $\gamma$ is minimal for the Finsler structure $F_t$. We have 
%\begin{align*}
%l(F_t^a)(\gamma)&=l(F)_t(\gamma)\\ 
%&=(1-t)l(F)(\gamma)+tl(F)(\gamma^{-1})\\
%&=(1-t)d(F)(x,y)+t d(F)(y,x)\\
%&=d(F)_t^a(x,y).
%\end{align*}
%
% It follows that $\gamma$ is minimal. A similar argument shows that $\gamma^{-1}$ is minimal as well.
\end{proof}

\cref{main-theorem} generalises a result of \cite{PT2} in which the Minkowski functional of the arithmetic symmetrisation of a Finsler metric is described. 

%Next, we turn to a more general setting, where we have two Finsler structures on a given manifold M.
%The first theorem deals with weighted sums of two Finsler structures.

The next  two theorems deal with the max of two Finsler structures.

\begin{theorem}\label{th:max}
	Let $F_1$ and $F_2$ be two  Finsler structures on $M$. Assume that for every $x$ and $y$ there exists an absolutely continuous path $\gamma_{x,y}: [0,1]\to M$
joining $x$ and $y$ with the following properties.

\begin{enumerate}[(1)]
	\item 
	$\gamma_{x,y}$ is a geodesic from  $x$ to $y$ with respect to  both $F_1$ and $F_2$.
	\item
	The function $(1-t)F_1(\gamma_{x,y}(s),\dot{\gamma}_{x,y}(s))-t F_2(\gamma_{x,y}(s),\dot{\gamma}_{x,y}(s))$, $s\in [0,1]$, is either non-negative or non-positive for almost all $s$, that is, it does not change sign on the interval $[0,1]$.
\end{enumerate}
 Then the metric 
$$\max\{(1-t)d(F_1),t d(F_2)\}$$
 is a Finsler metric 
with Lagrangian 
$$\max\{(1-t)F_1,t F_2\}.$$
 Furthermore, $\gamma_{x,y}$ is a geodesic with respect to the Finsler metric $\max\{(1-t)d(F_1), t d(F_2)\}$.
\end{theorem}
\begin{proof}
	Given $x,y\in M$ and $\epsilon>0$, there is a piecewise $C^1$-path $\alpha:[0,1]\to M$ such that 
	\begin{align*}
	d(\max\{(1-t)F_1,t F_2\})(x,y)
	&\geq \int_0^1 \max\{(1-t)F_1(\alpha(s),\dot{\alpha}(s)), tF_2(\alpha(s),\dot{\alpha(s)})\}\ ds -\epsilon\\
	  &\geq (1-t)\int_0^1F_1(\alpha(s),\dot{\alpha}(s)))-\epsilon\\ 
	  & \geq (1-t) d(F_1)(x,y)-\epsilon.
	 \end{align*}
 It follows that $d(\max\{(1-t)F_1,tF_2\})(x,y)
\geq (1-t)d(F_1)(x,y)$. Similarly we see that $d(\max\{(1-t)F_1,t F_2\})(x,y)
\geq t d(F_2)(x,y)$. Therefore we have 
$$d(\max\{(1-t)F_1,tF_2\})(x,y)\geq \max\{(1-t)d(F_1), t d(F_2)\}(x,y).$$

 Now we prove the reverse inequality. 
 Consider the path $\gamma=\gamma_{x,y}$ in the statement. We can assume without loss of generality that $$(1-t)F_1(\gamma(s),\dot{\gamma}(s))\geq t F_2(\gamma(s),\dot{\gamma}(s))$$
 for all $s\in [0,1]$, up to exchanging the roles of $F_1$ and $F_2$.
 Then by \cref{two lengths} and the Property (2), we have 
\begin{align*}
(1-t)d(F_1)(x,y)&=\int_0^1(1-t)F_1(\gamma(s),\dot{\gamma}(s))\ ds\\
&\geq \int_0^1t F_2(\gamma(s),\dot{\gamma}(s))\ ds\\
&=t d(F_2)(a,b).
\end{align*}
 Therefore $\max\{(1-t)d(F_1),t d(F_2)\}(x,y)=(1-t)d(F_1)(x,y)$. 
 Hence, again by \cref{two lengths}, we have 
\begin{equation}
\begin{split}
&\max\{(1-t)d(F_1),t d(F_2)\}(x,y)=(1-t)d(F_1)(x,y)\\
&=\int_0^1(1-t)F_1(\gamma(s),\dot{\gamma(s)})\ ds \\
&=\int_0^1\max\{(1-t)F_1(\gamma(s),\dot{\gamma}(s)),t F_2(\gamma(s),\dot{\gamma}(s)\}\ ds\\
&\geq d(\max\{(1-t)F_1,t F_2\})(x,y). 
\label{the second inequality}
\end{split}
\end{equation}

This proves the desired equality. Furthermore, the equality in \cref{the second inequality} shows that $\gamma_{x,y}$ is a geodesic with respect to the Finsler structure $\max\{(1-t)F_1,t F_2\}$.	
\end{proof}

We shall give now a necessary and sufficient condition for the equality between the two metrics appearing in \cref{th:max} to hold in the case when the Finsler structures are complete and $C^\infty$.

\begin{theorem}
\label{if and only if}
Let $F_1$ and $F_2$ be two  complete Finsler structures on $M$.
Then the Finsler metric induced by $\max\{(1-t)F_1, t F_2\}$ coincides with the metric $\max\{(1-t)d(F_1), td(F_2)\}$ if  the following condition (*) holds.
\smallskip

\noindent (*) For any two points $x, y$ in $M$, there is an absolutely continuous arc $\gamma$ joining $x$ and $y$ satisfying one of the three properties below.
\begin{enumerate}[(a)]
\item $\gamma$ is a geodesic with respect to  both $F_1$ and $F_2$, and $(1-t)F_1(\gamma(s),\dot{\gamma}(s))-tF_2(\gamma(s),\dot{\gamma}(s))$ is either non-negative or non-positive for almost all $s \in [0,1]$.
\item
$\gamma$ is a geodesic with respect to $F_1$, and $(1-t)F_1(\gamma(s), \dot{\gamma}(s)) \geq t F_2(\gamma(s), \dot{\gamma}(s))$ for almost all $s\in [0,1]$.
\item
$\gamma$ is a geodesic with respect to $F_2$, and $t F_2(\gamma(s), \dot{\gamma}(s)) \geq (1-t)F_1(\gamma(s), \dot{\gamma}(s))$ for almost all $s \in [0,1]$.
\end{enumerate}

If both $F_1$ and $F_2$ are complete $C^\infty$-Finsler structures, then (*) is also a necessary condition.
\end{theorem}

\begin{proof}
We first show the sufficiency of the conditions.

If Condition (a) holds, the conclusion follows from \cref{th:max}.

Suppose now that Condition (b) holds.
Then, since  $\gamma$ is an absolutely continuous geodesic with respect to $F_1$, we have, by \cref{two lengths},  $(1-t)d(F_1)(x,y)=(1-t)l_{F_1}(\gamma)$.
Since $(1-t)F_1(\gamma(s), \dot{\gamma}(s)) \geq t F_2(\gamma(s), \dot{\gamma}(s))$ all along $\gamma$ at almost every point, by \cref{longer}, we see that $(1-t)l_{F_1}(\gamma)\geq tl_{F_2}(\gamma)\geq t d(F_2)(x,y)$ .
Therefore, we have 
\begin{equation}
\label{max ineq}
\max\{(1-t)d(F_1)(x,y), t d(F_2)(x,y)\}=(1-t)l_{F_1}(\gamma).\end{equation}
%Thus we see that $d(\max\{(1-t)F_1, t F_2\})(x,y) \leq l(\max\{(1-t)F_1, t F_2\})(\gamma)$ and the last term is equal to $l(1-t)(F_1)(\gamma)$ by assumption.
%Thus we have $d(\max\{(1-t)F_1, t F_2\})(x,y)\leq \max\{(1-t)d(F_1)(x,y), t d(F_2)(x,y)\}$.

On the other hand since $\max\{(1-t)F_1, t F_2\}(x,v) \geq (1-t)F_1(x,v)$ for all $x$ and $v \in T_x(M)$, the minimality of $\gamma$ with respect to $F_1$ implies that for any arc $\delta $ joining $x$ and $y$, we have $l_{(\max\{(1-t)F_1, tF_2\})}(\gamma) = (1-t)l_{F_1}(\gamma) \leq (1-t)l_{F_1}(\delta) \leq l_{(\max\{(1-t)F_1, tF_2\})}(\delta)$.
Therefore,  by \cref{longer}, $\gamma$ is a geodesic also with respect to $\max\{(1-t)F_1, tF_2\}$ and  $d(\max\{(1-t)F_1, t F_2\})(x,y)=(1-t)l_{F_1}(\gamma)$.
Combining this with \cref{max ineq}, we see that $d(\max\{(1-t)F_1, t F_2\})(x,y)=\max\{(1-t)d(F_1),  d(F_2)\}$.
The same argument works also for Condition (c).

Now we turn to the necessity.
We shall show the contraposition.
Suppose that none of (a,b,c) in the statement hold.
%Since we assumed that $d(F_1)$ is complete, there is a geodesic $\gamma$ with respect to $d(F_1)$ which connects $x$ to $y$..
%Let $\gamma$ be a path joining $x$ and $y$ which is minimal with respect to $F_1$.

By \cref{hereditary}, we see that $\max\{(1-t)F_1, tF_2\}$ is also complete.
By \cref{smooth Hopf-Rinow}, there is a $C^\infty$-geodesic $\gamma$ with respect to $\max\{(1-t)F_1, tF_2\}$ joining $x$ to $y$.
Suppose first that $\gamma$ is also a geodesic for $F_1$.
Since neither (a) nor (b) holds, there is a non-trivial open interval $(a,b)$ in $[0,1]$ such that $tF_2(\gamma(s), \dot\gamma(s)) > (1-t)F_1(\gamma(s), \dot{\gamma}(s))$ for $s \in (a,b)$.
Then we have \begin{equation*}
\begin{split}
&l_{\max\{(1-t)F_1, t F_2\}}(\gamma)\\& \geq \int_a^b t F_2(\gamma(s), \dot\gamma(s))ds+\int_{[0,1]\setminus (a,b)} (1-t)F_1(\gamma(s), \dot{\gamma}(s))ds\\
&> \int_0^1 (1-t)F_1(\gamma(s), \dot{\gamma}(s)) ds=l_{((1-t)F_1)}(\gamma)=(1-t)d(F_1)(x,y).
\end{split}
\end{equation*}
If $\gamma$ is also a geodesic with respect to $F_2$, since (a) does not hold, by the same argument as above, changing the roles of $(1-t)F_1$ and $tF_2$, we see that $l_{\max\{(1-t)F_1, tF_2\}}(\gamma) > t d(F_2)(x,y)$.
If $\gamma$ is not a minimal path for $F_2$,  we have $l_{(\max\{(1-t)F_1, t F_2\})}(\gamma) \geq l_{(tF_2)}(\gamma) > t d(F_2)(x,y)$.
Therefore for such a path, we always have $l_{\max\{t F_1, (1-t)F_2\}}(\gamma)> \max\{td(F_1), (1-t)d(F_2)\}(x,y).$

By the same argument, and changing now $(1-t)F_1$ into $t F_2$, we see that for any path $\gamma$ connecting $x$ and $y$ that is minimal with respect to $F_2$, we have $l_{\max\{(1-t)F_1, tF_2\}}(\gamma)> \max\{(1-t)d(F_1), td(F_2)\}(x,y).$
If $\gamma$ is a geodesic for neither $F_1$ nor $F_2$,
then we have \begin{equation*}
\begin{split}
l_{\max\{(1-t)F_1, t F_2\}}(\gamma) &\geq \max\{l_{((1-t)F_1)}(\gamma), l_{(tF_2)}(\gamma)\}\\& > \max\{(1-t)d(F_1)(x,y), td(F_2)(x,y)\}.
\end{split}
\end{equation*}

Thus in all of  the three cases, we have 
\begin{equation*}
\begin{split}
d(\max\{(1-t)F_1, F_2\})(x,y)&=l_{\max\{(1-t)F_1, t F_2\}}(\gamma)\\ &>\max\{(1-t)d(F_1)(x,y), td(F_2)(x,y)\},
\end{split}
\end{equation*}
which is the negation of the assumption.
%Now, let $\{\gamma_i \colon [a,b] \to M\}$ be a sequence of $C^1$-paths joining $x$ to $y$ such that $\lim_{i \to \infty} l_{\max\{(1-t)F_1, t F_2\}}(\gamma_i) \to d(\max\{(1-t)F_1, tF_2\})(x,y)$.
%Since $d(F_1)$ is complete and $\{l(F_1)(\gamma_i)\}$ is bounded, bythe same argument as the proof of \cref{Hopf Rinow} using \cref{Arzela-Ascoli}, we see that $\{\gamma_i\}$ converges to a path $\gamma \colon [a,b] \to M$ uniformly with respect to $d(F_1)$.

%Since we assumed that $M$ is complete in $d(F_1)$ and $d(F_2)$, the paths connecting $x$ and $y$  with bounded lengths with respect to $F_1$ or $F_2$ form a compact set. This follows from a version of the Arzel\`a--Ascoli theorem for asymmetric metrics, se \cite{CZ}.
%Therefore, the inequalities which we have proved imply that $d(\max\{(1-t)F_1, t F_2\})(x,y) > \max\{(1-t)d(F_1), t d(F_2)\}(x,y).$
%This completes the proof of the necessity.
\end{proof}

\section{The Funk metric, the Hilbert metric and the weighted Funk Metrics}\label{s:Weighted-Funk}

Throughout this section, $\Omega$ is a closed convex subset of $\R^n$ with nonempty interior $\mathring{\Omega}$ and boundary $\partial \Omega$. 
We shall study two one-parameter families of weak metrics on $\Omega$. Both families are deformations of the Funk metric and one of them contains the Hilbert metric as a special case of the parameter. We first recall the definition of the Funk and Hilbert metrics on $\mathring{\Omega}$.
 
 Let us fix a real number $t\in [0,1]$. Given a pair of distinct points $x,y\in\mathring{\Omega}$, we denote the ray with origin $x$ and passing through  $y$ by $R(x,y)$. Let $a^+$ be the intersection point of $R(x,y)$ with the boundary $\partial \Omega$ of $\Omega$,  if this intersection is non-empty, that is, if $R(x,y)\not\subset \Omega$.  
(The convexity implies the uniqueness of the intersection point.)
Similarly let $a^-$ be the intersection point of $R(y,x)$ with $\partial \Omega$  if this point exists. 
The \emph{Funk metric} on $\Omega$ is defined, for $x\not=y$, by the formula
$$\mathcal{F}(x,y)=\left\{
\begin{array}{ll}
      \log\frac{\lvert x - a^+\rvert }{\lvert y-a^+\rvert}& R(x,y)\not\subset\Omega \\
      0 &   R(x,y)\subset \Omega . \\   
\end{array} 
\right.
$$
We define $\mathcal{F}(x,y)=0$ when $x=y$. 
The \emph{Hilbert metric} $\mathcal{H}$ is the arithmetic symmetrisation of the Funk metric $\mathcal{F}$, that is, it is defined by  $$\mathcal{H}(x,y)=\frac{1}{2}(\mathcal{F}(x,y)+\mathcal{F}(y,x)).$$
In particular, whenever both  $a^+$ and $a^-$ exist, the following formula holds: 
\begin{equation}\label{eq:Hilbert}\mathcal{H}(x,y)=\frac{1}{2}\log\bigg(\frac{\lvert x-a^+\rvert}{\lvert y-a^+\rvert}\frac{\lvert  y-a^-\rvert}{\lvert x-a^-\rvert}\bigg).
\end{equation}

Note that in \cref{eq:Hilbert} the quantity between parenthesis is the cross ratio of the ordered quadruple of points $(a^+, x, y, a^-)$,  which is a projective invariant, and that this makes the Hilbert metric a projectively invariant metric. More precisely, this  metric is invariant by the projective transformations of $\mathbb{R}^n$ that preserve  $\Omega$. In the present paper however, we shall not use this important fact. 

We consider the following two families of metrics, defined for $t\in [0,1]$: 
\begin{equation}\label{eqn1}
\mathcal{F}_t^a(x,y)=(1-t) \mathcal{F}(x,y)+t\mathcal{F}(y,x)
\end{equation}
\begin{equation}\label{eqn2}
\mathcal{F}_t^m(x,y)=\max\{(1-t)\mathcal{F}(x,y), t\mathcal{F}(y,x)\}.
\end{equation}

The Funk metric is a special case of \cref{eqn1} and \cref{eqn2}  obtained by setting $t=0$.
The Hilbert metric is a special case of \cref{eqn1} obtained by setting $t=1/2$.
Using the terminology we set before, we shall call a metric in one of the families  \cref{eqn1} and \cref{eqn2} a \emph{weighted} (\emph{arithmetic} and  \emph{max} respectively) \emph{Funk metric}.

\begin{proposition}\label{prop:weighted-geodesic}  For every $t\in[0,1]$, the straight lines are geodesics for the arithmetic weighted Funk metric $\mathcal{F}_t^a$.
\end{proposition}\label{prop;geodesics}
\begin{proof}
The proof can be done in  the same as for the case of the Hilbert metric, but we give a proof of a more general statement below (\cref{cor:straight-lines}).
%, and it is based on Figure  \ref{fig:a-geodesic}.
%In this figure, we have a Euclidean segment in the closure of the convex set $\Omega$, with the points $a^-,x,y,z, a^+$ in this order on the segment, and $a^-,a^+$ lie on $\partial \Omega$. Then, for any $t\in [0,1]$, we have
%\[\mathcal{F}_t^a(x,y)= (1-t)
%      \log\frac{\lvert x - a^+\rvert }{\lvert y-a^+\rvert}    
%      +t
%       \log\frac{\lvert y - a^-\rvert }{\lvert x-a^-\rvert},    
%\]
%\[\mathcal{F}_t^a(y,z)= (1-t)
%      \log\frac{\lvert y - a^+\rvert }{\lvert z-a^+\rvert}    
%      +t
%       \log\frac{\lvert z - a^-\rvert }{\lvert y-a^-\rvert},  
%\]
%and
%\[\mathcal{F}_t^a(x,z)= (1-t)
%      \log\frac{\lvert x - a^+\rvert }{\lvert z-a^+\rvert}    
%      +t
%       \log\frac{\lvert z - a^-\rvert }{\lvert x-a^-\rvert}.
%\]
%Adding and collecting terms, we obtain
%\[\mathcal{F}_t^a(x,y)+ \mathcal{F}_t^a(y,z)= 
%(1-t)\log\frac{\lvert x - a^+\rvert }{\lvert z-a^+\rvert} 
%+t \log\frac{\lvert z - a^-\rvert }{\lvert x-a^-\rvert} 
%=\mathcal{F}_t^a(x,z).
%\]
%This shows that the interior of the segment  is a geodesic in $\Omega$.
\end{proof}

\cref{prop:weighted-geodesic} gives a family of examples of metrics satisfying the condition of Hilbert's 4th problem, see \cite{Hilbert} for the statement   and \cite{Papa-Hilbert} for a recent overview of this problem. See also \cref{cor:straight-lines} below.

Likewise, we have the following analogue of a property known for the Hilbert metric.

\begin{proposition}\label{prop:uniqueness}  
For all $t\not=0$, the following two conditions are equivalent.
\begin{enumerate} 
\item The metric $\mathcal{F}_t^a$ is uniquely geodesic.
\item $\partial \Omega$ does not contain any pair of Euclidean segments that are contained in the same Euclidean plane and that are not collinear.
\end{enumerate}
\end{proposition}
The proof is the same as for the Hilbert metric given in \cite{Harpe}.

\begin{remark}\label{rk:weight}
By taking an appropriate example, we can see that in general the weighted max Funk metrics do not satisfy Proposition \ref{prop:weighted-geodesic}. See for instance the five ordered points $a^-, x,y,z,a^+$ depicted in Figure \ref{fig:m-geodesic}, for which we have $\mathcal{F}_{1/2}^m(x,y)= \log 1$, $\mathcal{F}_{1/2}^m(y,z)= \log 7/2$ and $\mathcal{F}_{1/2}^m(x,z)= \log 4$.
\end{remark}

	% First Diagram
%\begin{figure}[h!]
%		
%	\centering
%	\begin{tikzpicture}
%		% Draw the ellipse
%		\draw[thick] (0,0) ellipse (4cm and 2cm);
%		
%		% Draw the diameter line
%		\draw[thick] (-4,0) -- (4,0);
%		
%		% Add points and labels
%		\filldraw (-4,0) circle (2pt) node[left] {\(a^-\)};
%		\filldraw (4,0) circle (2pt) node[right] {\(a^+\)};
%		\filldraw (-1,0) circle (2pt) node[below] {$x$};
%		\filldraw (0,0) circle (2pt) node[below] {$y$};
%		\filldraw (1.5,0) circle (2pt) node[below] {$z$};
%		
%	\end{tikzpicture}
%	
%	\caption{Straight lines are geodesics for the arithmetic weighted Funk metric (\cref{fig:a-geodesic}). }
%	\label{fig:a-geodesic}
%\end{figure}

%\begin{figure}[p]
%	\hspace*{-4cm} 
%	\includegraphics[scale=1.0]{a-geodesic.pdf}
%	\caption{}
%	\label{fig:a-geodesic}
%\end{figure}

	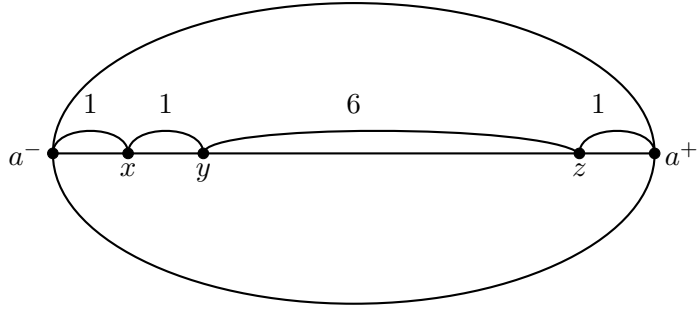
\begin{figure}[h!]
		
	\centering
	\begin{tikzpicture}
		% Draw the ellipse
		\draw[thick] (0,0) ellipse (4cm and 2cm);
		
		% Draw the diameter line
		\draw[thick] (-4,0) -- (4,0);
		
		% Add points and labels
		\filldraw (-4,0) circle (2pt) node[left] {\(a^-\)};
		\filldraw (4,0) circle (2pt) node[right] {\(a^+\)};
		\filldraw (-3,0) circle (2pt) node[below] {$x$};
		\filldraw (-2,0) circle (2pt) node[below] {$y$};
		\filldraw (3,0) circle (2pt) node[below] {$z$};
		
		% Add curved distance annotations
		\draw[thick] (-4,0) .. controls (-4,0.4) and (-3,0.4) .. (-3,0);
		\node[above] at (-3.5,0.4) {$1$};
		
		\draw[thick] (-3,0) .. controls (-3,0.4) and (-2,0.4) .. (-2,0);
		\node[above] at (-2.5,0.4) {$1$};
		
		\draw[thick] (-2,0) .. controls (-2,0.4) and (2.5,0.4) .. (3,0);
		\node[above] at (0,0.4) {$6$};
		
		\draw[thick] (3,0) .. controls (3,0.4) and (4,0.4) .. (4,0);
		\node[above] at (3.25,0.4) {$1$};
	\end{tikzpicture}
	\caption{In this figure, the numbers denote Euclidean length. They are chosen so that the line drawn is not a geodesic for the weighted max Funk metric with weight $1/2$ (\cref{rk:weight}).}
	\label{fig:m-geodesic}
\end{figure}
%\begin{figure}[p]
%	\hspace*{-4cm} 
%	\includegraphics[scale=1.2]{m-geodesic.pdf}
%	\caption{}
%	\label{fig:m-geodesic}
%\end{figure}

The proof of the following  proposition consists, like in \cref{rk:weight}, of taking an appropriate numerical example, and it is left to the reader.

%
%\begin{proposition}
% For all $t\in [0,1]$, any two points in the metric space $(\Omega, \mathcal{F}_t^m)$ can be joined by a piecewise linear geodesic, that is, a geodesic which is a concatenation of Euclidean segments.
%\end{proposition}
%

\begin{proposition}
Among the weak metrics in the two families $\mathcal{F}_t^a$ and $\mathcal{F}_t^m$ defined in \cref{eqn1} and  \cref{eqn2}, only those corresponding to the value $t=1/2$ are symmetric.
\end{proposition}

The  Funk and Hilbert metrics are classical examples of Finsler metrics and we shall recall their infinitesimal Finsler forms. Later in the paper, we shall consider weighted Funk and Hilbert metrics.

Let $\Omega$ be a  convex set in $\R^n$ with nonempty interior. For $x\in \mathring{\Omega}$  and for any nonzero vector $v\in \R^n=T_x\Omega$, we define 
the following function:
$$p(x,v)=\inf\{t>0\mid \frac{v}{t}+x \in \Omega\}.$$
Then $p$ is a Finsler structure on $\Omega$. It  is well known that the metric induced by this Finsler metric on $\Omega$ is the Funk metric \cite{Troyanov2014}. Note that if $a$ is the point on the boundary
of $\Omega$ which is on the ray with origin $x$ and direction $v$, then 

$$p(x,v)=\frac{\lvert v \rvert}{\lvert x-a \rvert }.$$

\noindent On the other hand, we set
$$q(x,v)=\frac{1}{2}(p(x,v)+p(x,-v)).$$
Then $q$ is a Finsler structure on $\mathring{\Omega}$ and the induced metric on $\Omega$ is the Hilbert metric. It is also known that if $a$ and $b$ are the intersection points of the line through $x$
in the direction of $v$ and the boundary of $\Omega$. Then
$$q(x,v)=\frac{\lvert v \rvert}{2}(\frac{1}{\lvert x-a\rvert }+\frac{1}{\lvert x-b\rvert}).$$

Now we prove  that the wighted arithmetic  Funk metrics are  Finslerian.

Let $\Omega$ be as before an open convex set in $\mathbb{R}^n$. 
\cref{main-theorem} gives the following corollary.
\begin{corollary}\label{cor:straight-lines}
The metric $\mathcal{F}_{t}^a$ defined in \cref{eqn1} is a Finsler metric with Lagrangian $p_t^a$, where $p$ is the Lagrangian of the Funk metric and 
$$p_t^a(x,v)=(1-t)p(x,v)+ t p(x,-v).$$
Furthermore straight line segments in $\Omega$ are minimal paths of $\mathcal{F}_{t,a}$.
\end{corollary}
\begin{proof}
It follows from  Theorem \ref{main-theorem} that $\mathcal{F}_{t}^a$ is a Finsler metric with Lagrangian $p_t^a$ and straight line segments are minimal paths.
\end{proof}

To conclude this section, we discuss two concrete examples.
The (Lagrangian of) the Funk metric in the unit ball \( B^n \subset \mathbb{R}^n \) is given by \cite[page 76]{Troyanov2014}
\[
p(x,v) = \frac{\langle x,v \rangle + \sqrt{ (1-\|x\|^2)\|v\|^2+\langle x,v \rangle^2}}{1-\|x\|^2}.
\]
 We therefore have
\[
p_t^a(x,v) = \frac{(1-2t)\langle x,v \rangle + \sqrt{ (1-\|x\|^2)\|v\|^2+\langle x,v \rangle^2}}{1-\|x\|^2},
\]

\bigskip

For the upper half-space $\mathcal{H} = \{x\in \mathbb{R}^n \mid x_n > 0\}$, we have
\[
p(x,v) = \max \left(\frac{v_n}{x_n}, 0\right), 
\]

We therefore have
\[
p_t^a(x,v) = \begin{cases} t\dfrac{v_n}{x_n}, & \text{if } v_n>0 \\  (1-t)\dfrac{|v_n|}{x_n}, & \text{if } v_n<0, \end{cases}
.\]

 	\section{Finsler metrics on Teichm\"uller spaces} \label{s:Teich}

In this section, we show how the ideas we discussed in this section fit in the study of metrics on Teichm\"uller spaces.

The first class of examplesspaces that we consider in this section are  Randers metrics  on Teichm\"uller.
% 
%
%There is an important class of Finsler metrics on a manifold $M$ which can be written as
%$$
%  F(x,v)  = F_0(x,v)  + \alpha_x (v),
%$$
%where $F_0$ is a reversible Finsler metric, that is, it satisfies  $F_0(x,-v)  = F_0(x,v) $ for all $v$,  and where $\alpha$ is a 1-form on $M$ such that $| \alpha_x (v)| < F_0(x,v)$ for any $(x,v) \in TM$. In particular, if $F_0$ is the square root of a Riemannian metric $g$, then $F$ is said to be a \textit{Randers metric}.
%For such a metric, we clearly have
%\begin{equation*}
% F_t^a(x,v) = F_0(x,v)  + (1-2t) \alpha_x (v),
%\end{equation*}
%for each $t\in[0,1]$,
%and
%%\footnote{\textcolor{red}{I cannot find a nice formula for $F_t^m(x,v)$ here when $t\neq \frac{1}{2}$. Perhaps it is not so difficult; the point is to reduce an expression of the type
%%$ \max  \big( t\cdot (A+B) , (1-t) \cdot (A-B) \big)$ to a simpler form.}}
%\begin{equation*}
%  F_t^m(x,v)= \max  \big( t\cdot (F_0(x,v)  + \alpha_x (v)) , (1-t) \cdot (F_0(x,v)  - \alpha_x (v)) \big).
%\end{equation*}
%In particular, setting $t= 1/2$ we have
%\begin{equation*}
% F^a_{\frac{1}{2}}(x,v) = F_0(x,v),
%\end{equation*}
%and 
%\begin{equation*}
% F^m_{\frac{1}{2}}(x,v) =  \frac{1}{2} \left( F_0(x,v) +  |\alpha_x(v)|\right) .
%\end{equation*}

A Randers metric is originally associated with an $n$-dimensional Riemannian manifold $(M,g)$ and a differential  1-form $\omega$ on $M$ satisfying $\|\omega\|_g<1$ at every point of $M$. It is defined infinitesimally as the Finsler metric associated with the differential form   $F(v)=g(v,v)^{1/2}+\omega(v)$. Randers metrics have applications in physics. In their original form, they were introduced in \cite{Randers}. They were later studied by various authors, especially 
in the case when the initial metric on $M$ is Finsler and not Riemannian,   

In the paper \cite{MOP-2022}, we introduced on the Teichm\"uller space of the torus a natural family of Randers metrics connecting the Teichm\"uller metric to the Thurston asymmetric metric of the Teichm\"uller space (adapted to the case of the torus).  
  We gave a description of the unit tangent circle of each of these metrics.

Motivated by this case of the torus, we studied in the  
 paper \cite{MOP-2024} a family of (asymmetric) Randers metrics which are deformations of the Teichm\"uller metric on the Teichm\"uller space $ \mathcal{T}_{g,n}$ of a hyperbolic orientable surface $S_{g,n}$ of genus $g$ with $n$ punctures (possibly with $n=0$),  where the given differential 1-form is the differential of the logarithm of the extremal length function associated with a measured foliation.
We showed that this metric, which we call the \emph{Teichm\"uller--Randers metric}, is not complete on any Teichm\"uller disc, and we gave a characterisation of geodesic rays with bounded length in this disc in terms of their directing measured foliations.

 We consider now a family of metrics on spaces of Euclidean triangles and triangulated surfaces.

Let   $\frak{T}_1$ be the set of Euclidean triangles of unit area.
We recall the definition of an asymmetric metric  $\eta$  on $\frak{T}_1$
we introduced in the paper \cite{SOP}. 

Each triangle is marked by a choice of a homeomorphism between this triangle and a fixed disc with three marked points on the boundary, the homeomorphism sending the marked points to the vertices of the triangle.
Let $a_1, a_2, a_3>0$ denote the lengths of the  edges of a marked triangle. We 
 parametrise the set of marked Euclidean triangles by the following subset of $\R^3$:
$$\{(a_1,a_2,a_3) : a_1,a_2,a_3>0, \ a_2+a_3-a_1>0, a_1+a_3-a_2>0, a_1+a_2-a_3>0\}$$

This set is identified with the product space $(\R^*_+)^3=\{(A_1,A_2,A_3): A_1, A_2, A_3 >0\}$ via the mapping

$$A_1=\frac{a_2+a_3-a_1}{2}$$
$$A_2=\frac{a_3+a_1-a_2}{2}$$
$$A_3=\frac{a_1+a_2-a_3}{2}.$$

The area of a triangle $(a_1,a_2,a_3)$ in terms of the parameters $(A_1,A_2,A_3)$ is given by Heron's formula:
\begin{equation}\label{f:Heron}
\mathrm{Area}(a_1,a_2,a_3)=\mathrm{Ar}(A_1,A_2,A_3)=\sqrt{(A_1+A_2+A_3)A_1A_2A_3}.
\end{equation}

Now consider the function $$\eta: (\R^*_+)^3\times (\R^*_+)^3\to \R$$ defined by

$$\eta((A_1,A_2,A_3),(A_1',A_2',A_3'))=\log \max \{A_1'/A_1 ,A'_2/A_2, A_3'/A_3\}.$$ 

There is a natural action of $\R^*_+$ on $(\R^*_+)^3$ by $\lambda(A_1,A_2,A_3)=(\lambda A_1,\lambda A_2, \lambda A_3)$. This action corresponds to scaling a triangle by the factor $\lambda$. 

For $X,Y \in (\R^*_+)^3$ and for $\lambda,\lambda'\in \mathbb{R}$, we have

$$\exp(\eta(\lambda X,\lambda'Y))=\frac{\lambda'}{\lambda}\exp(\eta(X,Y)).$$

Finally,  the  subspace $\frak{T}_1\subset (\R^*_+)^3$ consisting of unit area triangles is the set:

$$\frak{T}_1=\{(A,B,C): \mathrm{Ar}(A,B,C)=1\}.$$

This function $\eta$ defines an asymmetric metric on $\frak{T}_1$. 	In the paper \cite{SOP}, we studied  this metric. In particular, we proved that it is Finsler and we gave a characterization of its geodesics. Then we studied the same questions for  a metric on a space of surfaces equipped with singular flat metrics relative to a fixed triangulation.  We examined its geodesics, its Finsler structure and its completeness. We also addressed questions on deformations and one-parameter families of such a metric.

One such deformation is the following. Consider the family of metrics on $\frak{T}_1$ parametrised by $t\in [0,1]$ and defined for each such $t$ by 
$$\eta^a_t(X,Y)=(1-t)\eta(X,Y)+t\eta(Y,X),$$
where $X,Y \in \frak{T}_1$. As usual, the superscript $a$ in $\eta^a_t$ stands for ``arithmetic". We ask the following questions on these metrics, which are not answered in the paper \cite{SOP}.
   
\begin{enumerate}

 \item Are all these metrics Finsler ?

\item  Are all these metrics non-symmetric except for $t=1/2$?

\item Are any two metrics $\eta^a_t$ and $\eta^a_{t'}$ are non-isometric for $t\not=t'$ ?

\end{enumerate}
One may ask similar questions for the metrics on $\frak{T}_1$ defined by 

$$\eta^m_t(X,Y)=\max\{(1-t)\eta(X,Y),t\eta(Y,X)\}.$$
where, as before,  the superscript $m$ in $\eta^m_t$ stands for $\max$.

\medskip
Finally, we make a relation with the paper \cite{PTh}. It was proved in that paper that  Thurston's asymmetric metric satisfies Busemann's axiom, and this fact was used to show that Thurston's metric is both forward and backward complete and also geodesically complete in both directions.
Our \cref{Finsler Busemann} gives an alternative proof of this in a much more general setting.
We can also apply our result to the earthquake metric studied in \cite{HOPP}.
\cref{Finsler Busemann} implies the equivalence of the forward and backward topologies, which was proved in Appendix A of \cite{HOPP}.

\noindent
\.{I}smail Sa\u{g}lam\, Adana Alparslan Turkes Science and Technology University,\\ Department of Aerospace Engineering, 
Adana,T\"{u}rkiye

\noindent e-mail: isaglamtrfr@gmail.com

\medskip

\noindent Ken'ichi Ohshika,
Department of Mathematics,
Gakushuin University,
Mejiro, Toshima-ku, 171-8588 Tokyo, Japan

\noindent email: 
ohshika@math.gakushuin.ac.jp

\medskip

\noindent Athanase Papadopoulos, 
 Institut de Recherche Math\'ematique Avanc\'ee, 
 CNRS et Universit\'e de Strasbourg, 
7 rue Ren\'e Descartes, 67084, Strasbourg, France

 \noindent  email: 
papadop@math.unistra.fr\\

\end{document}